\documentclass[preprint]{elsarticle}
\usepackage[ruled,vlined]{algorithm2e}
\usepackage{graphicx} 
\usepackage{epstopdf}

\biboptions{numbers,sort&compress} 
\usepackage{bm}
\usepackage{bbold}
\usepackage{amsmath}
\usepackage{amsthm}
\usepackage{xcolor}
\usepackage[cal=boondox]{mathalfa}
\usepackage{lipsum}
\usepackage{amsfonts}
\usepackage{epstopdf}
\usepackage{algorithmic}
\usepackage{mathrsfs}
\usepackage[caption=false]{subfig}
\usepackage{mathtools}

\usepackage{chngcntr}
\usepackage{apptools}

\newcommand{\diag}{\mathop{\mathrm{diag}}}
\newcommand{\argmin}{\mathop{\mathrm{arg} ~\mathrm{min}}}

\newcommand{\R}{{\mathbb{R}}}

\newtheorem{hypothesis}{Hypothesis}
\DeclareMathOperator{\Tr}{Tr}

\newtheorem{theorem}{Theorem}[section]

\newtheorem{lemma}[theorem]{Lemma}
\newtheorem{remark}{Remark}
\begin{document}

\begin{frontmatter}

\title{Bi-level algorithm for optimizing hyperparameters in penalized nonnegative matrix factorization}


\author[inst1]{Nicoletta Del Buono}
%\ead{nicoletta.delbuono@uniba.it}
\author[inst1]{Flavia Esposito}
%\ead{flavia.esposito@uniba.it}
\author[inst1]{Laura Selicato\corref{mycorrespondingauthor}}
\cortext[mycorrespondingauthor]{Corresponding author}

\affiliation[inst1]{organization={Department of Mathematics, University of Bari Aldo Moro},
            addressline={Via Orabona 4},
            city={Bari},
            postcode={70125}, 
            country={Italy}}
\ead{laura.selicato@uniba.it}

\author[inst2]{Rafał Zdunek}

\affiliation[inst2]{organization={Faculty of Electronics, Photonics, and Microsystems, Wroclaw University of Science and Technology}, 
            addressline={27 Wybrzeze Wyspianskiego st.},
            city={Wrocław},
            postcode={50370}, 
            country={Poland}}
%ead{rafal.zdunek@pwr.edu.pl}

\begin{abstract}

Learning approaches rely on hyperparameters that impact the algorithm's performance and affect the knowledge extraction process from data. Recently, Nonnegative Matrix Factorization (NMF) has attracted a growing interest as a learning algorithm. This technique captures the latent information embedded in large datasets while preserving feature properties. NMF can be formalized as a penalized optimization task in which tuning the penalty hyperparameters is an open issue. The current literature does not provide any general framework addressing this task. This study proposes to express the penalty hyperparameters problem in NMF in terms of a bi-level optimization. We design a novel algorithm, named Alternating Bi-level (AltBi), which incorporates the hyperparameters tuning procedure into the updates of NMF factors. Results of the existence and convergence of numerical solutions, under appropriate assumptions, are studied, and numerical experiments are provided.

\end{abstract}

\begin{keyword}
Nonnegative Matrix Factorization,  Hyperparameter Optimization, Penalty coefficient, Low-rank approximation
%\texttt{elsarticle.cls}\sep \LaTeX\sep Elsevier \sep template
\MSC[2010] 
15A23 \sep 65K10 \sep 65F55 \sep 68Q32 \sep 68V20 \sep 90C46 \sep 46N10
\end{keyword}
\end{frontmatter}

%\linenumbers
\section{Introduction}
All learning models require setting some hyperparameters
(HPs){\let\thefootnote\relax\footnote{{Abbreviations - HP: Hyperparameter - HPO: Hyperparameter Optimization - \\GB: Gradient-based - MU: Multiplicative Updates - 
RMD: Reverse-Mode Differentiation - 
FMD: Forward-Mode Differentiation -
P-MU: Penalized Multiplicative Update.}}}-- variables governing the learning approach -- before starting the learning process from data.
HPs tuning requires a substantial effort, depending on the user, and affects the learner's performance~\cite{k2}. Automatic Hyperparameter Optimization (HPO) would bring a solution to these problems~\cite{bes14}.

HPO strategies commonly used in the literature range from simple methods, such as the grid or random search, to more complex ones, such as the Bayesian optimization or the Genetic Algorithms (GAs)~\cite{bergstra2012random,Francescomarino2018GeneticAF,bergstra2011algorithms,Marinov2019HyperparameterOW,Hussain_et_al2021,Guoxin_Yong_BayesianHPO2020}. Grid search explores a prescribed set of HPs in a given search space, while random search defines a random sampling of HPs without any assumption on the search space. Both these strategies are time-consuming since they are driven by some performance metrics, commonly measured by cross-validation. Moreover, they require domain experts to justify a search space that is meaningful for the application domain. Bayesian optimization attempts to predict how unseen combinations of HPs will perform based on a so-called surrogate model that approximates the HPO problem. GAs are based on stochastic optimization and are inspired by the biological phenomena of natural evolution.  Recently, some works proved that Gradient-Based (GB) approaches can obtain great results in HPO for large-scale problems, using only local information and at least one HP (learning rate) \cite{nostro,6789800}.
%Practically, trial-and-error approaches via Grid or Random Search perform HPs selection; these methods are time-consuming since they are driven by some performance metric, commonly measured by cross-validation or they require a domain expert to justify the value of some HP within a meaning belonging to the application domain.}\\  
%Recently, the scientific community has developed approaches ranging from Black-Box to Bayesian or Gradient-Based (GB). 
GB methods reduce the validation error, computing or approximating the gradient with respect to HPs~\cite{bott98,Bott3,Mc}. One of the ways to go through GB methods for HPO is to formalize the problem as a bi-level task~\cite{fra2,Pedro,del2021toward}. Bi-level programming solves an outer optimization problem subject to the optimality of an inner optimization one~\cite{bard2013practical}.

Formally, let $\mathcal{A}$ be a learner with hyperparameter vector $\bm{\lambda} \in \mathbb{R}^p$, parameter vector\footnote{$\mathbf{w}$ can be a scalar, a vector or a matrix.} $\mathbf{w} \in \mathbb{R}^q$, with $p, q \in \mathbb{N}$, and $\mathbf{X} \in \R^{n \times m}$ with $n,m \in \mathbb{N}$ an assigned data matrix. For learning model $\mathcal{A}$, the HPO can be written  as:

\begin{equation} \label{eq1}
\bm{\lambda}^*=\argmin_{\bm{\lambda} \in \Lambda}\mathscr{F}(\mathcal{A}(\mathbf{w}(\bm{\lambda}),\bm{\lambda}), \mathbf{X})
\quad 
\text{s. t.} \quad \mathbf{w}(\bm{\lambda}) = \argmin\limits_{\mathbf{w}} \mathscr{L}(\bm{\lambda}, \mathbf{X}),
\end{equation}

where $\mathscr{F}$ evaluates how good is $\mathbf{w}$ gained by learner $\mathcal{A}$ tuned with hyperparameter $\bm{\lambda}$ on $\mathbf{X}$, and $\mathcal{L}$ is an empirical loss. Typically, the inner problem aims to minimize empirical loss $\mathscr{L}$; the outer problem is related to HPs. Because of the implicit dependence of the outer problem on $\bm{\lambda}$, equation \eqref{eq1} is challenging to solve.  Recently, first order bi-level optimization techniques based on estimating Jacobian $\frac{d\mathbf{w}(\bm{\lambda})}{d\bm{\lambda}}$ via implicit or iterative differentiation have been proposed to solve \eqref{eq1}~\cite{franceschi2017forward,Mc,Pedro}.

However, there are still no effective results of using GB methods for HPO in the unsupervised field. This study aims to use these techniques to revise problem \eqref{eq1} in an unsupervised learning context, to automatically achieve HPs. %Thus, this study combines these approaches to tackle the HPO associated with matrix decomposition in an unsupervised scenario. 
We consider Nonnegative Matrix Factorization (NMF) and its constrained variants (in particular sparseness constraint)~\cite{cichocki2007multilayer,chu2021alternating,esposito2019orthogonal,gillis2020nonnegative,kim2007sparse,kim2008nonnegative,lin2007projected,liu2017regularized,merritt2005interior}. We regard these problems as penalized optimization tasks in which penalty coefficients are HPs, focusing on their proper choice via HPO. Taking advantage of the bi-level HPO problem formulation, we construct an alternating bi-level approach that includes the HPs choices as a part of the algorithm that computes the factors in the NMF data approximation task under study.  The rest of this section reviews preliminary concepts on NMF and its sparsity constraints with the importance of the penalty HPs for sparse NMF. Section \ref{proposal}  describes the novel bi-level formulation of the penalized NMF and its treatment via an alternating methodology.   
We prove the existence of the solution to this problem, and we use convergence results to design a new algorithm, named Alternating Bi-level (AltBi), which is described in Section \ref{algoritmo}. It is our numerical proposal to solve the HPO issue in NMF models with additional sparsity constraints. Section \ref{experiments} illustrates the numerical results obtained using the AltBi algorithm on synthetic and real signal datasets. Section \ref{conclusion} sketches some conclusive remarks and future works.

\subsection{Preliminaries}
NMF groups some methodologies aiming to approximate nonnegative data matrix $\mathbf{X}\in\mathbb{R}_+^{n\times m}$ %(\textcolor{red}{typically} $n>>m$, \textcolor{red}{otherwise in the unsupervised context we consider the transposed matrix}) 
as $\mathbf{X}\approx \mathbf{WH}$, where $\mathbf{W}\in\mathbb{R}_+^{n\times r}$ is the \textit{basis} matrix, and $\mathbf{H}\in\mathbb{R}_+^{r\times m}$ is the \textit{encoding} (or \textit{coefficient}) matrix.
The choice of parameter $r$, which determines the number of rows of $\mathbf{H}$ (respectively, columns of $\mathbf{W}$) and $r <<\min(n,m)$, is problem-dependent and user-specified; and it represents an example of HP connected with NMF. A general NMF problem can be formulated as an optimization task 

\begin{equation}
\min\limits_{\mathbf{W}\geq 0, \mathbf{H}\geq 0}{D_\beta(\mathbf{X},\mathbf{WH})}= \min\limits_{\mathbf{W}\geq 0, \mathbf{H}\geq 0} \sum\limits_{i=1}^n{\sum\limits_{j=1}^m}{d_{\beta}(x_{ij},\sum\limits_{k=1}^r w_{ik}h_{kj})},%(WH)_{ij}
\label{optprobl}
\end{equation}

where the objective function $D_\beta(\cdot,\cdot)$ is a $\beta$-divergence assessing how well its reconstruction $\mathbf{WH}$ fits $\mathbf{X}$, where $d_\beta$ is generally defined for each $x,y \in \R$ as

\begin{equation*}
d_{\beta}(x,y)= \left \{
\begin{array}{lc}
\frac{1}{\beta(\beta-1)}(x^{\beta}+(\beta-1) y^{\beta}-\beta xy^{\beta-1}) &\beta\in \mathbb{R} \setminus \{0,1\};\\
x\log(\frac{x}{y})-x+y &\beta=1;\\
\frac{x}{y}-\log(\frac{x}{y})-1 &\beta=0.\\
\end{array}
\right.
\end{equation*}

Either data properties and specific application domain influence the particular choice of $D_\beta$ (popular measures are for $\beta = 2,1,0$, i.e., the Frobenius norm, the generalized Kullback-Leibler (KL) and the Itakura-Saito (IS) divergences, respectively).\\
The NMF model in (\ref{optprobl}) can also be enriched with additional constraints by introducing penalty terms

\begin{equation}
\min\limits_{\mathbf{W}\geq 0, \mathbf{H}\geq 0}{D_\beta(\mathbf{X},\mathbf{WH}) + \lambda_\mathbf{W} \mathcal{R}_1(\mathbf{W})+ \lambda_\mathbf{H} \mathcal{R}_2(\mathbf{H})},
\label{eq145}
\end{equation}

where $\mathcal{R}_1:\mathbb{R}^{n\times r}\rightarrow\mathbb{R}$ and $\mathcal{R}_2:\mathbb{R}^{r\times m}\rightarrow\mathbb{R}$ are some penalty functions enforcing specific properties on the factor matrices; $\lambda_\mathbf{W}, \lambda_\mathbf{H} \in \R_+$ are the penalty coefficients (i.e., HPs), that balance the bias-variance trade-off in approximating $\mathbf{X}$ and preserving the additional constraints. It is assumed that at least one of the two HPs is non-null for the penalty to make sense, and (\ref{eq145}) allows to penalize simultaneously one or both factors. The problem of properly selecting the penalty HPs is still an unsolved issue in constrained NMF.

One example of a suitable constraint to impose on NMF factors is sparsity. Sparseness leads to several advantages; it allows obtaining some form of compression, improves the computational cost and gives us better interpretability when many features (the columns in $\mathbf{X}$) are present, and the model becomes very large. Several zeros avoid over-fitting, allow a way for feature extraction, and elude modeling the noise implicitly embedded in the data. 
%NICOLETTA riformulerei la seguente
%Nonnegativity in NMF algorithms naturally yields sparse factors. Nevertheless, \textcolor{red}{given} the uncontrollability of the factor sparseness degree\textcolor{red}{, it is preferred to \sout{and} use} direct constraints \textcolor{red}{that} can enforce this property ~\cite{zheng2009tumor}. 
Nonnegativity in the NMF algorithms naturally produces sparse factors. Nonetheless, because the factor sparseness degree is uncontrollable, it is preferable to use direct constraints that can enforce this property~\cite{zheng2009tumor}.
Various penalty terms enforce sparsity in NMF: an example is to apply $\ell_0$ “norm” on $\mathbf{W}$ and $\mathbf{H}$~\cite{gao2015hyperspectral}. % \textcolor{red}{\sout{It originates from the Lasso problem~\cite{tibshirani1996regression} and addresses several computational issues in Machine Learning and Pattern Recognition.}} 
However, this penalization makes the associated objective function non-smooth, globally non-differentiable, and non-convex, resulting in an NP-hard optimization problem \eqref{eq145}. Conversely, due to their analytical proprieties, %\footnote{$\ell_1$ norm is a convex, non smooth, global non differentiable function, whereas $\ell_2$ norm is a convex, smooth, global differentiable function.}, 
$\ell_1$ and $\ell_2$ norms are valid alternatives to $\ell_0$~\cite{zhang2015survey}. In particular, $\ell_1$ norm originates from the Lasso problem~\cite{tibshirani1996regression} and addresses several computational issues in machine learning and pattern recognition.
Sparsity can also be imposed via $\ell_{1,2}$ norm which is used either as a penalty function or as an objective function~\cite{kong2011robust,li2013dictionary,nie2010efficient}.
The Hoyer's sparse NMF optimization task uses the normalized ratio of $\ell_1$ and $\ell_2$ norm computed on the columns of $\mathbf{W}$ and rows of $\mathbf{H}$~\cite{hoyer2004non}. Section \ref{algoritmo} illustrates our algorithm proposal to tune HPs using an objective function based on the KL-divergence and $\ell_1$ norm.

\subsection{The penalty HP in NMF} 
Usually, static optimization mechanisms, such as the grid or random search, perform HPs tuning in constrained NMF \eqref{eq145}. These approaches solve several variants of the same problem associated with a predefined discrete set of HPs and then choose the best one according to empirical criteria, (an example can be found in the context of gene expression analysis~\cite{kimpark_sparseNMF2007}). Other approaches are based on the Discrepancy Principle (DP) and the L-curve criterion which are empirical methods used to tune the penalty value in Tikhonov regularization~\cite{hansen1992analysis,hansen1993use}. Active-set approaches for the NMF model, which are based on the Frobenius norm and the Tikhonov regularization on $\mathbf{W}$, are other sophisticated strategies for tuning penalty HPs, and they usually choose the best penalty HP according to clustering performance~\cite{zdunek2014regularized}. Bayesian optimization methodologies are exploited to solve the problem 

\begin{equation}
    \min\limits_{\mathbf{H}\geq 0}{\frac{1}{2}||\mathbf{X}-\mathbf{WH}||_F^2+\frac{\lambda}{2}\mathcal{R}(\mathbf{H})},
    \label{eq:froregH}
\end{equation}

where $\mathcal{R}(\mathbf{H})=\Tr{(\mathbf{H}^\top \mathbf{EH})}= \sum\limits_i ||\mathbf{h}_i||_1^2 = ||\mathbf{H}||_{2,1}^2$, $\mathbf{h}_i$ is the $i$-th row of $\mathbf{H}$, $\mathbf{E}\in\mathbb{R}^{r\times r}$ is the all-ones matrix that enforces sparsity on $\mathbf{H}$'s columns using the squared norm $\ell_{2,1}$, and $\lambda \in \R$ is the penalty HP~\cite{zdunek2007nonnegative}. %This is an alternative approach to the norm $\ell_{1,2}$. 
In the associated minimization problem, the choice of $\lambda$ is made according to the following exponential rule

\begin{equation}
    \lambda^{(k)}=\lambda_0 \exp{(-\tau k)},
\end{equation}

where $k$ is the number of iteration in the algorithm, $\lambda_0$ is the initial value of the HP and $\tau$ is a parameter controlling the results.\\
In this study, we want to automate the choice of HPs through GB methods and bi-level approach in order to free the HPs tuning from the domain expert and any empirical or cross-validation related techniques.

\section{New Formulation} 
\label{proposal}
Several approaches can tackle HPO in model \eqref{eq145}, even though a uniform theory applicable to general objectives and penalty functions is still lacking. The results reported in this study aim to fill this void. \\This section presents the main contribution of the work. We reformulate the model \eqref{eq145} as:

\begin{equation}\label{new}
   \min\limits_{\mathbf{H}\geq 0,\mathbf{W}\geq 0} D_{\beta}(\mathbf{X}, \mathbf{WH})+ \mathcal{R}_1(\mathbf{L_W} \mathbf{W})+\mathcal{R}_2(\mathbf{H L_H}),
\end{equation}

where $\mathbf{L_W} \in \R^{n \times n}$ and $\mathbf{L_H} \in \R^{m \times m}$ are diagonal matrices of HPs associated with each row of $\mathbf{W}$ and each column of $\mathbf{H}$, respectively, and $\mathcal{R}_1: \R^{n \times r} \to \R$ and $\mathcal{R}_2: \R^{r \times m} \to \R$ are the penalty functions being continuous and such that $\mathcal{R}_{i}(\mathbb{0})=0$ for each $i= 1,2$, where $\mathbb{0}$ is the zero element in $\mathbb{R}^{n \times r}$ and $\mathbb{R}^{r \times m}$, respectively. In this way, each row and each column are penalized independently. Although the problem can be written for both factors $\mathbf{W}$ and $\mathbf{H}$, we, for now, focus on the case where
%The results obtained in the following are referred to the case
$\mathbf{L_H}=\mathbb{0}_{\R^{m \times m}}$ and $\mathbf{L_W}=\mathbf{L}\in\mathbb{R}^{n\times n}$, diagonal and %such that $\mathbf{L} \neq \mathbb{0}_{\R^{n \times n}}$ 
non-null matrix (because the penalty makes sense) so that \eqref{new} is reduced to 

\begin{equation}\label{partcase}
    \min\limits_{\mathbf{H}\geq 0,\mathbf{W}\geq 0} D_{\beta}(\mathbf{X}, \mathbf{WH}) +  \mathcal{R}(\mathbf{L W}).
\end{equation}

A symmetric extension can be easily derived for $\mathbf{L_H} \neq \mathbb{0}_{\R^{m \times m}}$ and $\mathbf{L_W} = \mathbb{0}_{\R^{n \times n}}$.  On the other hand, simultaneous optimization on both factors (with respect to \eqref{new}) requires some supplement theory related to the bi-level formulation of NMF for columns, which will be the subject of future works. %that will be the subject of future works. 
Problem \eqref{partcase} is convex in each variable separately\footnote{for particular values of $\beta$ and specific penalty functions.}.
Alternating optimization techniques are helpful to incorporate into the minimization process the updates of each NMF factor separately. Firstly, fixing $\mathbf{W}$, one estimates $\mathbf{H}$; subsequently,  $\mathbf{H}$ is fixed to estimate $\mathbf{W}$. To tune the penalty HP matrix $\mathbf{L}$, we incorporate it simultaneously into the process of updating factor $\mathbf{W}$, introducing a bi-level strategy on each row of $\mathbf{W}$. 

Let $\mathbf{w}_i \in \R^r$ be the $i$-th column of $\mathbf{W}^\top$ and from now on, let $\bm{\lambda} \in \R^n$ indicate the vector of diagonal elements of $\mathbf{L}$ and $\lambda_i \in \Lambda \subset \R$ the $i$-th diagonal element of $\mathbf{L}$. 
We first consider the simple minimization problem in $\mathbf{H}$ (for fixed $\mathbf{W}$):

\begin{equation} 
\label{one}
    \min\limits_{\mathbf{H} \geq 0} D_{\beta}(\mathbf{X}, \mathbf{WH}).
\end{equation}

To obtain the update for $\mathbf{W}$ and achieve an optimal solution for matrix $\mathbf{L}$, we use the bi-level task applied to each $\mathbf{w}_i$, i.e the $i$-th row of $\mathbf{W}$, which for each $i = 1,\dots, n$ reads:

\begin{eqnarray}\label{two}
\min\limits_{\lambda_i \in \Lambda} f(\lambda_i), \qquad 
  f(\lambda_i) = \inf \{\mathcal{E}(\mathbf{w}_i(\lambda_i), \lambda_i): \mathbf{w}_i(\lambda_i) \in \argmin\limits_{\mathbf{u} \in \mathbb{R}^{r}}\mathscr{L}_{\lambda_i}(\mathbf{u})\},
  \end{eqnarray}
  
where $f: \Lambda \to \R$ is the so-called {\it Response Function} (RF) of the outer problem related to the $i$-th row of $\mathbf{W}$ (according to the bi-level notation). Namely, first we fix an outer level hyperparameter $\lambda_i$, then we solve the inner level problem finding $\mathbf{w}_i$ as argmin of a loss function. Finally, the feasible solution of $\min_{\lambda_i \in \Lambda} f(\lambda_i)$ is evaluated. Note that the RF associated with the entire matrix problem is $F(\bm{\lambda}) = \sum\limits_{i=1}^n f(\lambda_i)$. 
%\emph{Error Function} 
{\it Error Function} (EF) $\mathcal{E}$ is the outer objective such that

\begin{equation}\label{error}
        \mathcal{E}:\mathbb{R}^{r} \times \Lambda \to \mathbb{R}: (\mathbf{w}_i, \lambda_i) \mapsto \sum\limits_{j=1}^m d_{\beta}(\mathbf{x}_j ,\sum\limits_{k=1}^r w_{ik}(\lambda_i) h_{kj}),
\end{equation} 

where for every $\lambda_i \in \Lambda$; whereas {\it Loss Function} (LF) $\mathscr{L}_{\lambda_i}$  is the inner objective

\begin{equation}\label{loss}
    \mathscr{L}_{\lambda_i} : \mathbb{R}^{r} \to \mathbb{R}: \mathbf{w}_i \mapsto \sum\limits_{j=1}^m d_{\beta}(\mathbf{x}_j ,\sum\limits_{k=1}^r w_{ik} h_{kj})+ \lambda_i \mathcal{r}(\mathbf{w}_i),
\end{equation}

where $\mathcal{r}: \R^r \to \R$ is a linear function, closely related to the enforcement of the constraint, such that $\sum\limits_{i=1}^n \lambda_i \mathcal{r}(\mathbf{w}_i)= \mathcal{R}(\mathbf{LW})$. \\
In the following section, we clarify how to handle each part of optimization problems \eqref{one} and \eqref{two}. 

\subsection{Finding the unpenalized factor}
\label{penaH}
To solve \eqref{one}, different update rules satisfying diverse requirements exist (fast convergence or easy implementation mechanisms); they range from multiplicative to additive update rules~\cite{berry2007algorithms,wang2013nonnegative}. 
In this study, we focus on the standard NMF Multiplicative Updates (MU)~\cite{lee2000algorithms} due to their ease of implementation and monotonic convergence. From initial matrices, MU uses scaling rules from the minimization of an auxiliary function (derived from Richardson-Lucy or Expectation-Maximization (EM) approaches~\cite{dempster1977maximum,lange1984reconstruction,lucy1974iterative,richardson1972bayesian,saul1997aggregate}). 
%Even if some slow convergence properties represents the main drawback of any approach based on auxiliary function, this approach is often used to solve NMF problems because it ensures the nonnegativity of the computed factors without further handling \cite{fevotte2011algorithms,lee2000algorithms}.
Any approach based on an auxiliary function is often used to solve NMF problems because it ensures the nonnegativity of the computed factors without further handling, notwithstanding that it converges slowly~\cite{fevotte2011algorithms,lee2000algorithms}.

We briefly review the update rule for the general $\beta$ divergence giving the particular result for the KL divergence.\\
Considering the update rule

\begin{equation}\label{generalizeupdata}
    \mathbf{H} \leftarrow \mathbf{H} .* \frac{\mathbf{W}^\top ((\mathbf{WH})^{\cdot[\beta-2]}\cdot* \mathbf{X})}{\mathbf{W}^\top (\mathbf{WH})^{\cdot [\beta-1]}},
\end{equation}

being $.*$ the Hadamard product (exponential and ratio operators are computed element-wise), it is known that the general $\beta$-divergence $D_{\beta}(\cdot, \cdot)$, is non-increasing using rule \eqref{generalizeupdata} for $0 \leq \beta \leq 2$. In particular, the paper \cite{lee1999learning} shows this result for $\beta =2$ and $\beta = 1$. In \cite{kompass2007generalized}, it is generalized to the case $1 \leq \beta \leq 2$. In practice, we observe that the criterion is still non-increasing under update \eqref{generalizeupdata} for $\beta < 1$ and $\beta > 2$ (and in particular
for $\beta = 0$, that corresponds to the IS divergence). More details on theoretical results and proofs can be found in~\cite{fevotte2009nonnegative,fevotte2011algorithms,lee2000algorithms}.\\
Specifically for the KL divergence, \eqref{generalizeupdata} becomes:

\begin{equation}\label{H_KL}
%W \leftarrow W.*\frac{(X./(W H))H^\top}{1_n * (\sum_{j=1}^m H_{:j})^\top}, \qquad 
\mathbf{H} \leftarrow \mathbf{H}.*\frac{\mathbf{W}^\top(\mathbf{X}./(\mathbf{W H}))}{(\sum\limits_{i=1}^n \mathbf{w}_i) \cdot \mathbb{1_m}^\top},
\end{equation}

where $\mathbb{1}_m$ is the ones-vector of dimension $m$.\\

\subsection{Finding the penalized factor and solving the HPO}\label{subsection22}
To obtain the update for $\mathbf{W}$ and determine an optimal solution for penalty matrix $\mathbf{L}$, we use bi-level approach \eqref{two} applied on each row of $\mathbf{W}$. To simplify the notation, from now until the end of subsection \ref{subsection22}, subscript $i$ for $\mathbf{w}_i$ and $\lambda_i$ is omitted.
For the sake of simplicity, we suppose the existence of a unique minimizer $\mathbf{w}_{(\lambda)}$ for the inner objective. Nevertheless, problem \eqref{two} generally has no closed expression for $\mathbf{w}_{(\lambda)}$, so it does not allow to optimize the outer objective function  directly. 

A reliable approach is to replace the inner problem with a dynamical system~\cite{franceschi2017forward,franceschi2021unified,Mc}. This point allows us to compute an exact gradient of an approximation of \eqref{two}. It also enables optimization of the HPs that define the learning dynamics. As mentioned before, depending on how the gradient with respect to HPs is calculated, two main approaches can be used: the implicit differentiation, based on the implicit function theorem, and the iterative differentiation
approach. In this work, we will focus on the latter.\\
Therefore, the solution of the inner object minimization as a dynamical system with state $\mathbf{w}^{(t)} \in \mathbb{R}^{r}$ can be written as:

\begin{equation} \label{Veneziaa 1}
\mathbf{w}^{(t)}= \Phi_t(\mathbf{w}^{(t-1)}, \lambda) \quad t=1,.,T;  
\end{equation}

%with initial condition $\mathbf{w}^{(0)} = \Phi_0(\lambda)$, where for every $t=1,\dots,T$ function $\Phi_t :(\mathbb{R}^{r} \times \R) \rightarrow \mathbb{R}^{r}$ is a smooth mapping. In our case, this function represents the update performed by the $t$-th step of the optimization algorithm since $\Phi_t$ is a row-wise update for $\mathbf{W}$. 
with initial condition $\mathbf{w}^{(0)} = \Phi_0(\lambda)$, where $\Phi_t :(\mathbb{R}^{r} \times \R) \rightarrow \mathbb{R}^{r}$ 
is a smooth map, and it is the row-wise update for $\mathbf{W}$, for $t=1,\dots, T$. %This function is the $t$-th step update of the optimization algorithm being $\Phi_t$ a row-wise update for $\mathbf{W}$. 
Note that $\mathbf{w}^{(t)}$ for all $i=1,\dots, n$ %,\dots, \mathbf{w}^{(T)}$ implicitly 
depend on $\lambda$, implicitly. 

Bi-level problem \eqref{two} can be approximated (for each $i= 1,\dots,n)$ using the constrained procedure:

\begin{equation}\label{min2}
       \min\limits_{\lambda} f(\lambda) \quad \text{s. t.} \quad 
       \mathbf{w}^{(t)}=\Phi_t(\mathbf{w}^{(t-1)}, \lambda) \quad \text{for}  \quad t=1, \dots, T.
\end{equation}

In general, procedure \eqref{min2} might not be the best approximation for bi-level problem \eqref{two} since the minimizer of $\mathscr{L}_{\lambda}$, to which the optimization dynamic converges, does not necessarily minimize $\mathcal{E}$.
% because the optimization dynamic converges to some minimizer of inner objective $\mathscr{L}_{\lambda}$, but not necessarily to the one that also minimizes function $\mathcal{E}$. 
%The scenario is different if the inner problem admits a unique minimizer for every $\lambda \in \Lambda \subset \R$, as we will see in detail in the next subsection \ref{exi}. 
This problem is overcome by assuming the uniqueness of the minimizer of $\mathscr{L}_{\lambda}$, for any $\lambda \in \Lambda \subset \R$, as we will see in detail in the following subsection \ref{exi}.
Moreover, we note that for $1 \leq \beta \leq 2$, thanks to the convexity of the $\beta$ divergence function, and consequently, of $\mathscr{L}_{\lambda}$\footnote{$\mathscr{L}_{\lambda}$ is convex as a sum of convex functions.}, the associated problems $\argmin f^{(T)}(\lambda)$, $\argmin f(\lambda)$, and $\argmin \mathscr{L}_{\lambda}$ are singleton, where $f^{(T)}$ is the response function at time $T$.

\subsubsection{Existence and Convergence Results}\label{exi}
We provide results on the existence of solutions to problem \eqref{two} and the (variational) convergence for approximate problem \eqref{min2} related to it.
%%%%%%%%%%%%%%%%%%%%%%%%%%%%%%%%
\begin{hypothesis}\label{Hp}
Considering the following assumptions:
\begin{enumerate}
    \item $\Lambda \subset \R$ is compact; %subset of $\mathbb{R}$;
    \item Error Function \eqref{error} %$\mathcal{E} : \mathbb{R}^{r} \times \Lambda \to \mathbb{R}$ 
    is jointly continuous\footnote{The function is continuous with respect to each variable separately.};
    \item application $(\mathbf{w}, \lambda) \to \mathscr{L}_{\lambda}(\mathbf{w})$ is jointly continuous, and problem $\argmin \mathscr{L}_{\lambda}$ is a singleton for every $\lambda \in \Lambda$;
    \item $\forall \lambda \in \Lambda$, $\mathbf{w}_{(\lambda)} = \argmin{\mathscr{L}_{\lambda}}$ is bounded.
\end{enumerate}
\end{hypothesis}
Therefore, bi-level problem \eqref{two} can be reformulated as follows:

\begin{equation}\label{sette}
    \min_{\lambda \in \Lambda}f(\lambda) = \mathcal{E}(\mathbf{w}_{(\lambda^*)}, \lambda^*), \qquad \mathbf{w}_{(\lambda)}=\argmin_\mathbf{u} \mathscr{L}_{\lambda}(\mathbf{u}),
\end{equation}

where $(\mathbf{w}_{(\lambda^*)},\lambda^*)$ is the optimal solution.

\begin{theorem}[Existence]
Problem \eqref{sette} admits solutions under the assumptions $1-4$. %in \ref{Hp}.
\end{theorem}
\begin{proof}
From the compactness of $\Lambda$, the continuity of $f$ ensures minimizers exist.
%it follows that a sufficient condition for the existence of minimizers is that $f$ is continuous. Consider $\hat{\lambda} \in \Lambda$ and let $(\lambda_n)_{n\in \mathbb{N}}$ be a sequence in $\Lambda$ such that $\lambda_n \to \hat{\lambda}$. 
Consider $\hat{\lambda} \in \Lambda$ and sequence $(\lambda_n)_{n\in \mathbb{N}}$ in $\Lambda$ converging to $\hat{\lambda}$. 
Due to the boundness of associate sequence $(\mathbf{w}_{(\lambda_n)})_{n \in \mathbb{N}}$, there is a converging subsequence $(\mathbf{w}_{(\lambda_{k_n})})_{n \in \mathbb{N}}$ such that $\lim\limits_{\lambda_{k_n} \to \hat{\lambda}} \mathbf{w}_{(\lambda_{k_n})}= \hat{\mathbf{w}} \in \R^r$.\\
%Since associated sequence $(\mathbf{w}_{(\lambda_n)})_{n \in \mathbb{N}}$ is bounded, there exists a subsequence $(\mathbf{w}_{(\lambda_{k_n})})_{n \in \mathbb{N}}$ such that $\mathbf{w}_{(\lambda_{k_n})} \to \hat{\mathbf{w}}$ for some $\hat{\mathbf{w}} \in \R^{r}$. 
For point $3$ in Hypothesis \ref{Hp}, since $\lambda_{k_n}$ converges to $\hat{\lambda}$, it results:%the map $(\mathbf{w}, \lambda) \mapsto \mathscr{L}_{\lambda}(\mathbf{w})$ is jointly continuous, we have

\begin{equation}
    \forall \mathbf{w} \in \R^{r} \quad \mathscr{L}_{\hat{\lambda}}(\hat{\mathbf{w}})= \lim_n \mathscr{L}_{\lambda_{k_n}}(\mathbf{w}_{(\lambda_{k_n})}) \leq \lim_n \mathscr{L}_{\lambda_{k_n}}(\mathbf{w})= \mathscr{L}_{\hat{\lambda}}(\mathbf{w}).
\end{equation}

Thus, $\hat{\mathbf{w}}$ is a minimizer of $\mathscr{L}_{\hat{\lambda}}$ and consequently $\hat{\mathbf{w}} = \hat{\mathbf{w}}_{(\lambda)}$. This proves that sequence $(\mathbf{w}_{(\lambda_n)})_{n \in \mathbb{N}}$ is bounded and has a unique accumulation point.\\ Consequently $(\mathbf{w}_{(\lambda_n)})_{n \in \mathbb{N}}$ converges to $\mathbf{w}_{(\hat{\lambda})}$ (i.e. its unique accumulation point). \\Lastly, for point 2 of Hypothesis \ref{Hp} and since $(\mathbf{w}_{(\lambda_n)}, \lambda_n) \to (\mathbf{w}_{(\hat{\lambda})}, \hat{\lambda})$, it follows 
%\begin{equation*}
   $f(\lambda_n)=\mathcal{E}(\mathbf{w}_{(\lambda_n)}, \lambda_n) \to \mathcal{E}(\mathbf{w}_{(\hat{\lambda})}, \hat{\lambda}) = f(\hat{\lambda})$,
%\end{equation*}
that concludes the proof.
\end{proof}

\begin{theorem}[Convergence]\label{Connv} In addition to Hypothesis \ref{Hp}, %in \ref{Hp},
suppose that:
\begin{itemize}
    \item[5.] $\mathcal{E}(\cdot, \lambda)$ is uniformly Lipschitz continuous;
    \item[6.] $(\mathbf{w}^{(T)}_{(\lambda)})_{T \in \mathbb{N}} \to \mathbf{w}_{(\lambda)}$ %the iterates $(\mathbf{w}^{(T)}_{(\lambda)})_{T \in \mathbb{N}}$ converge uniformly to $\mathbf{w}_{(\lambda)}$ 
    uniformly on $\Lambda$ for $T \to + \infty$.
\end{itemize} 
Then
\begin{itemize}
    \item[(a)] $\inf f^{(T)}(\lambda) \to \inf f(\lambda)$,
    \item[(b)] $\argmin f^{(T)}(\lambda) \to \argmin f(\lambda)$.%, meaning that, for every
%$(\lambda_T)_{T\in \mathbb{N}}$ such that $\lambda_T \in \argmin f^{(T)}(\lambda)$ , we have that:
%\begin{itemize}
%    \item $(\lambda_T)_{T \in \mathbb{N}}$ admits a convergent subsequence;
%    \item for every subsequence $(\lambda_{K_T} )_{T \in \mathbb{N}}$ such that $\lambda_{K_T} \to \hat{\lambda}$, we have $\hat{\lambda} \in \argmin f(\lambda)$.
%\end{itemize}
\end{itemize}
\end{theorem} 
To prove Theorem \ref{Connv}, the following preliminary result concerning the stability of minima and minimizers in optimization problems is helpful (the complete proof of this result can be found in~\cite{dont}). 

\begin{theorem}
\label{the37}
Let $g_T$ and $g$ be lower semi-continuous functions defined on a compact set $\Lambda$. If $g_T \to g$ uniformly on $\Lambda$ for $T \to +\infty$, then
\begin{itemize}
    \item[(a)] $\inf g_T \to \inf g$
    \item[(b)] $\argmin g_T \to \argmin g$. %, meaning that, for every $(\lambda_T )_{T \in \mathbb{N}}$ such that $\lambda_T \in \argmin g_T$ , we have that:
    %\begin{itemize}
      %  \item $(\lambda_T)_{T \in \mathbb{N}}$ admits a convergent subsequence;
      %  \item for every subsequence $(\lambda_{K_T})_{T \in \mathbb{N}}$ such that $\lambda_{K_T} \to \hat{\lambda}$, we have $\hat{\lambda} \in \argmin g$.
    %\end{itemize}
\end{itemize}
\end{theorem}
Thanks to these results, Theorem \ref{Connv} can be proved.
\begin{proof}[Proof of Theorem \ref{Connv}]
%Since $\mathcal{E}(\cdot, \lambda)$ is uniformly Lipschitz continuous, there exists $\nu > 0$ such that for every $T \in \mathbb{N}$ and every $\lambda \in \Lambda$ results:

The uniform Lipschitz continuity of $\mathcal{E}(\cdot, \lambda)$ ensures that there exists $\nu > 0$ such that:

\begin{equation*}
    |f^{(T)}(\lambda)-f(\lambda)|= |\mathcal{E}(\mathbf{w}^{(T)}_{(\lambda)}, \lambda)- \mathcal{E}(\mathbf{w}_{(\lambda)}, \lambda)| \leq \nu ||\mathbf{w}^{(T)}_{(\lambda)}- \mathbf{w}_{(\lambda)}||,
\end{equation*}

for every $T \in \mathbb{N}$, and $\lambda \in \Lambda$.\\
Since $\mathcal{E}(\cdot, \lambda)$ is uniformly Lipschitz continuous, it results that $f^{(T)}(\lambda) \to f(\lambda)$ uniformly on $\Lambda$ as $T \to +\infty$. The thesis follows by Theorem \ref{the37}.
\end{proof}
Hypotheses $1-6$ are satisfied by many problems of practical interest, in particular when $1 \leq \beta \leq 2$. Results for other values of $\beta$ could be obtained, losing the hypothesis of convexity of $f$ and $\mathscr{L}_{\lambda}$.

%It is worthy to underline that the hypotheses $(1)-(6)$ are satisfied by many problems of practical interest, in particular when $1 \leq \beta \leq 2$. For other values of $\beta$ we are studying these results even if we lose the hypothesis of convexity of $f$ and $ \mathscr{L}_{\lambda} $. However, we are confident that we will find similar existence and convergence results since for $\beta < 1$ and $\beta >2$ $ f$ and $\mathscr{L}_{\lambda}$ can be rewritten by the sum of concave and convex functions.

\subsubsection{Solving the bi-level problem} \label{RMD}
Bi-level problem \eqref{sette} or approximate problem \eqref{min2} satisfy existence and convergence theorems, respectively, so we can focus on  finding penalty HPs matrix $\mathbf{L}$ in practice. We now reintroduce subscript $i$ for $\mathbf{w}_i$ and $\lambda_i$. Applying a gradient type approach on each diagonal element of $\mathbf{L}= \diag(\bm{\lambda})$, the optimization of $\bm{\lambda}$ depends on the approximation of hypergradient $\nabla_{\bm{\lambda}}F$. Using the chain rule, it results: %Since it is not directly dependent on $\bm{\lambda}$, we can derive it using the chain rule:

\begin{equation}\label{hypergradient}
     \frac{\partial F}{\partial \lambda_i} = \frac{\partial f}{\partial \lambda_i} + \frac{\partial f}{\partial \mathbf{w}_i^{(T)}} \cdot \frac{d \mathbf{w}_i^{(T)}}{d \lambda_i}, \quad \forall i= 1, \dots, n,
\end{equation}

where $\frac{\partial f}{\partial \lambda_i} \in \R$ and $\frac{\partial f}{\partial \mathbf{w}_i^{(T)}} \in \R^{r}$ are available.

Following the iterative differentiation approach, the computation of the hypergradient can be done using the Reverse-Mode Differentiation (RMD)  or Forward-Mode Differentiation (FMD). RMD computes the hypergradient by back-propagation; instead, FMD works with forwarding propagation. In our algorithm, we use only the second mode; for completeness, we report both.

\paragraph{\textbf{Reverse Mode}}
The reverse strategy to compute the hypergradient is based on the Lagrangian perspective calculated for \eqref{min2}, that is $\mathfrak{L}:\mathbb{R}^r \times \Lambda \times \mathbb{R}^r \to \mathbb{R}$ which is defined as %in the following
%\begin{equation}\label{Venezia 6}
    $ \mathfrak{L}(\mathbf{w}_i,\lambda_i, \bm{\alpha}) 
     = %\sum\limits_{j=1}^m d_{\beta}(x_j, \sum\limits_{a=1}^r w^{(T)}_a H_{a j})
     \mathcal{E}(\mathbf{w}_i^{(T)}, \lambda_i)
     + \sum \limits_{t=1}^T\bm{\alpha}_t^\top(\Phi_t(\mathbf{w}_i^{(t-1)},\lambda_i)-\mathbf{w}_i^{(t)})$ 
%\end{equation}
for $i = 1, \dots, n$, where, for each $t =1,\dots,T$, $\bm{\alpha}_t \in \mathbb{R}^r$ are the Lagrange multipliers associated with the $t$-th step of the dynamics. The partial derivatives of Lagrangian $\mathfrak{L}$ are

\begin{equation*}
%\begin{align}
\frac{\partial{\mathfrak{L}}}{\partial{\bm{\alpha}_t}}=\Phi_t(\mathbf{w}_i^{(t-1)},\lambda_i)-\mathbf{w}_i^{(t)}, \quad
\frac{\partial{\mathfrak{L}}}{\partial{\mathbf{w}_i^{t}}} =\bm{\alpha}^\top_{t+1}\mathbf{A}_{t+1}-\bm{\alpha}^\top_t,
\end{equation*}

\begin{equation*}
\frac{\partial{\mathfrak{L}}}{\partial{\mathbf{w}_i^{(T)}}} = 
\nabla \mathcal{E}(\mathbf{w}_i^{(T)}, \lambda_i)-\bm{\alpha}^\top_T, \quad
\frac{\partial{\mathfrak{L}}}{\partial{\lambda_i}} = \sum\limits_{t=1}^T \bm{\alpha}^\top_t \mathbf{b}_t,
%\end{align}
\end{equation*}
where 

\begin{equation}\label{Aandb}
    \mathbf{A}_t = \frac{\partial{\Phi_t(\mathbf{w}_i^{(t-1)}, \lambda_i)}}{\partial{\mathbf{w}_i^{(t-1)}}} \in \R^{r \times r} \quad \text{and} \quad \mathbf{b}_t =  \frac{\partial{\Phi_t(\mathbf{w}_i^{(t-1)}, \lambda_i)}}{\partial{\lambda_i}} \in \R^{r \times 1}. 
\end{equation}

Therefore the optimality conditions give the iterative rules of RMD: 

\begin{equation}
\begin{cases}
    \bm{\alpha}_T^\top = \nabla \mathcal{E}(\mathbf{w}_i^{(T)}, \lambda_i), \\  h_T = \frac{\partial f}{\partial \lambda_i},\\
    h_{t-1} = h_t + \mathbf{b}_t\bm{\alpha}^\top_t,\\ 
    \bm{\alpha}^\top_{t-1} = \mathbf{A}_t\bm{\alpha}^\top_t,
\end{cases}
\end{equation}

for $t=T, \dots, 1$ and $i = 1, \dots, n$. Then the $i$-th component of the hypergradient can be computed as $\frac{\partial f}{\partial \lambda_i}  = h_{0}$. 
\paragraph{\textbf{Forward-Mode}}  
FMD computes the derivative of \eqref{hypergradient} by the chain rule.  
Each $\Phi_t$ depends on $\lambda_i$ directly,  and on $\mathbf{w}_i^{(t-1)}$ indirectly, for $t=1,\dots,T$. Hence: 
%Applying again the chain rule, 
%For every $t=1,\dots,T$, the derivative of the state with respect to $\lambda_i$ is:

\begin{equation} \label{Venezia 13}
\frac{d \mathbf{w}_i^{(t)}}{d \lambda_i}= \frac{\partial \Phi_t(\mathbf{w}_i^{(t-1)},\lambda_i)}{\partial \mathbf{w}_i^{(t-1)}}\frac{d \mathbf{w}_i^{(t-1)}}{d \lambda_i}+\frac{\partial \Phi_t(\mathbf{w}_i^{(t-1)}, \lambda_i)}{\partial \lambda_i}.
\end{equation}

Defining $\mathbf{s}_t= \frac{d \mathbf{w}_i^{(t)}}{d \lambda_i} \in \R^{r}$, each FMD iterate is:

\begin{equation}\label{eqdiff}
\begin{cases}
    \mathbf{s}_0 = \mathbf{b}_0; \\ \mathbf{s}_t = \mathbf{A}_t \mathbf{s}_{t-1}+ \mathbf{b}_t \quad t=1, \dots, T;
\end{cases}
\end{equation}

where $\mathbf{A}_t$ and $\mathbf{b}_t$ are defined as above, and the $i$-th component of the hypergradient is 

\begin{equation}\label{Hypergradiente_f}
    \frac{\partial F}{\partial \lambda_i}= \mathbf{g}_T^\top \cdot \mathbf{s}_T \in \R, \quad
%\end{equation*}
\text{being} \quad
%\begin{equation*}
  \mathbf{g}_T = \frac{\partial f}{\partial \mathbf{w}_i^{(T)}} \in \R^{r}. 
\end{equation}

Letting $\mathbf{s}_0=0$, the solution of \eqref{eqdiff} solves: %actually the solution of the more general difference equation:

\begin{equation} \label{hyppp}
\frac{\partial F(\lambda_i)}{\partial \lambda_i} = \frac{\partial f^{(T)}(\lambda_i)}{\partial \mathbf{w}_i^{(T)}} \Big(\mathbf{b}_T + \sum_{t=0}^{T-1}(\prod_{s=t+1}^T \mathbf{A}_s) \mathbf{b}_t\Big) .
\end{equation}

\paragraph{\textbf{Computational considerations}}
Opting between RMD and FMD depends on balancing the trade-off based on the size of $\mathbf{w}_i$ and $\lambda_i$. The RMD approach requires that $\mathbf{w}_i^{(t)}$ for all $i = 1, \dots, n$ and all $t= 1, \dots, T$ are stored in memory to compute $\mathbf{A}_t$ and $\mathbf{b}_t$ in the backward pass, and therefore it is suitable when the quantity $rT$ is small.
As we will see later, our approach uses the FMD strategy that requires time $O(rT)$ and space $O(r)$ for every row and iteration.

\section{Alternating Bi-level  Algorithm - AltBi}\label{algoritmo}
In this section, we present our Alternating Bi-level (AltBi) algorithm for the particular case of $\beta=1$ and $\ell_1$ as penalty function in \eqref{partcase}:

\begin{equation}\label{partcase_ex}
        \min\limits_{\mathbf{H}\geq 0,\mathbf{W}\geq 0} D_{1}(\mathbf{X}, \mathbf{WH}) + ||\mathbf{L W}||_1.
\end{equation}
It implements the procedures described in the previous sections, performing NMF updating, including the automatic setting of the HPs.  As its name suggests, AltBi optimizes $\mathbf{H}$ and $\mathbf{W}$ alternately through the bi-level approach. 
Sub-interval of arbitrary length $T$, called $bunch$, is considered to perform the bi-level procedure on $\mathbf{W}$. It ensures the extraction of a convergent sub-sequence from any bounded sequence\footnote{This holds for the Bolzano-Weierstrass Theorem.}. Even if this is not unique, it is enough to consider its sub-sequence to have the same limit. 

Algorithm \ref{algor} shows the pseudo-code for AltBi. It receives as input data matrix $\mathbf{X}$, the rank of factorization $r$, initial matrices $\mathbf{W}$, $\mathbf{H}$, and vector $\bm{\lambda}$ of the diagonal elements of $\mathbf{L}$. We initialize the number of iterates $MaxIter$, tolerance $tol$, and length $T$ of the bunch. The outer while-loop repeats the alternating algorithm until one of the two conditions $err>tol$ or $iter<MaxIter$ is false. Error is defined as the absolute value of the difference in the divergence calculated between two successive iterates divided by the divergence at the initial step. The inner for-loop performs the bi-level procedure for every bunch, calculating the hypergradient to update $\bm{\lambda}$ with a gradient method. The algorithm returns the optimal matrices $\mathbf{W}^*$, $\mathbf{H}^*$, and $\mathbf{L}^* = \diag(\bm{\lambda}^*)$. 

\begin{algorithm}[h]
\caption{Alternating Bi-level Algorithm - AltBi}
\label{algor}
\SetAlgoLined
\KwData{$\mathbf{X} \in \R_{+}^{n \times m}$,  $r < \min(n,m)$.}
\KwResult{$\mathbf{W}^{*} \in \R_{+}^{n \times r}$,  $\mathbf{H}^{*} \in \R_{+}^{r \times m}, \mathbf{L}^{*}=\diag(\bm{\lambda}^{*}) \in \R_{+}^{n \times n}$.}
\textbf{Initializations:} $\mathbf{W} \in \R_{+}^{n\times r}$, $\mathbf{H} \in \R_{+}^{r \times m}$,  $\mathbf{L}=\diag(\bm{\lambda} = (\lambda_1, \dots, \lambda_n))$, $T$, $MaxIter$, $tol$, $err$, and $iter$.\\ 
\While{(err $>$ $tol$) \& (iter $<$ $MaxIter$)}{
%$\mathbf{H} = \text{update}(\mathbf{X},\mathbf{W},\mathbf{H})$ as in \eqref{H_KL}\;
update $\mathbf{H}$ as in \eqref{H_KL}\;
\For{$t \in \{1, \dots,T\}$}{
    \For{$i \in \{1, \dots,n\}$}{
   % $\mathbf{w}_i^{(t)} = \text{update}(\mathbf{H},\mathbf{x}_i, \lambda_i)$ as in \eqref{multidivKL1}\;
   update $\mathbf{w}_i^{(t)}$ as in \eqref{multidivKL1}\;
    compute $\mathbf{A}_t$ and $\mathbf{b}_t$ as in \eqref{Aandb}\;
    compute $\frac{\partial F}{\partial \lambda_i^{(t)}}$ as in \eqref{hyppp}\;}}
    
    %$(\mathbf{w}_i^{(t)}, \frac{\partial F}{\partial \lambda_i^{(t)}})= \text{bi-level}(\mathbf{X}, \frac{\partial F}{\partial \lambda_i^{(t-1)}}, \mathbf{w}_i(\lambda_i)^{(t-1)},\mathbf{H})$ as in \eqref{multidivKL1} and \eqref{hyppp}\;}}
    
%$\bm{\lambda} = \text{update}(\bm{\lambda} , \nabla_{\bm{\lambda}}F)$ as in \eqref{uplambda}\;
rearrange $\mathbf{w}_i$ for all $i = 1, \dots, n$ to construct $\mathbf{W}$\; 
rearrange $\frac{\partial F}{\partial \lambda_i^{(t)}}$ for all $i = 1, \dots, n$ to construct $\nabla_{\bm{\lambda}}F$\;   
update $\bm{\lambda}$ as in \eqref{uplambda}\;
$\text{iter}+ = 1$\; 
}
\end{algorithm}

Referring to \eqref{partcase_ex}, we use the MU rule specified in \eqref{H_KL} for updating $\mathbf{H}$. For the bi-level formulation, we keep the KL divergence with the $\ell_1$ norm as a loss function:

\begin{equation}
    \mathscr{L}_{\lambda_i} : \mathbb{R}^{r} \to \mathbb{R}: \mathbf{w}_i \mapsto \sum\limits_{j=1}^m d_{1}(\mathbf{x}_j,\sum\limits_{k=1}^r w_{ik} h_{kj})+ \lambda_i ||\mathbf{w}_i||_1,
\end{equation}

whereas the KL divergence is the error function of the outer problem:

\begin{equation}
        \mathcal{E}:\mathbb{R}^{r} \times \Lambda \to \mathbb{R}: (\mathbf{w}_i, \lambda_i) \mapsto \sum\limits_{j=1}^m d_{1}(\mathbf{x}_j,\sum\limits_{k=1}^r w_{ik}(\lambda_i) h_{kj}).
\end{equation}

To assess the theoretical results for the previous functions, Hypothesis \ref{Hp} needs to be verified. Observe that  $\mathcal{E}$ is jointly continuous with respect to $\mathbf{w}_i$ and $\lambda_i$. Similarly, $d_{1}$ and the map $(\mathbf{w}_i, \lambda_i) \mapsto \mathscr{L}_i(\mathbf{w}_i)$. From the convexity and the compactness of $\Lambda$, $\argmin \mathscr{L}_{{\lambda}_i}$ is a singleton for any $\lambda_i$. Finally, ${\mathbf{w}_i}_{(\lambda_i)} = \argmin \mathscr{L}_{\lambda_i}$ remains bounded as $\lambda_i$ varies in $\Lambda$, in fact:

\begin{equation*}
    ||\mathbf{w}_i(\lambda_i)|| \leq M \quad \forall \lambda_i \in \Lambda \quad \text{with} \quad M >0, \quad M\leq M^*, 
\end{equation*}

being $M^*= \max \{||\mathbf{w}_i(\lambda_i)||_2^2, \quad \lambda_i \in \Lambda \}$.\\
Matrix $\mathbf{W}$ is updated using the following novel rule by rows $\Phi: \R^r \times \Lambda \to \R^r$ s.t. $(\mathbf{w}^{(t-1)}_i, \lambda_i) \mapsto \mathbf{w}^{(t)}_i$ (this update can be similarly derived as in~\cite{lee2000algorithms,liu2003non}), then for $k=1,\dots,r$ and $i=1, \dots, n$

\begin{equation}\label{multidivKL1}
w^{(t)}_{ik} = w^{(t-1)}_{ik} \frac{
\sum\limits_{j=1}^m{h_{kj} (x_{ij}/\sum\limits_{a=1}^r{w_{ia} h_{aj}})}}{\sum\limits_{j=1}^m{h_{kj}}+\lambda_i}. %\quad \text{}
\end{equation}

Its proof is detailed in %Appendix
\ref{appendix}.\\
Vector $\bm{\lambda}$ is also updated by the steepest descent procedure: 

\begin{equation}\label{uplambda}
    \bm{\lambda}=\bm{\lambda}-c\nabla_{\bm{\bm{\lambda}}}F(\bm{\lambda}), %\quad s=1,\dots,S,
\end{equation}

with stepsize $c = \frac{1}{\textit{iter}}$\footnote{The usual conditions on the stepsize are fulfilled:
$\sum\limits_{s=1}^{MaxIter} c_s = \infty \quad \text{and} \quad \sum\limits_{s=1}^{MaxIter} c_s^2 < \infty$.}. \\
Each component of the hypergradient in \eqref{Hypergradiente_f} can be expressed with 

\begin{equation*}
{({{\mathbf{g}}_T}^\top)}_k = - \sum\limits_{j=1}^m (\frac{x_{ij}}{\sum\limits_{a=1}^r w_{ia}^{(T)} h_{aj}} h_{kj} +h_{kj}) \quad \text{for $k=1,\dots, r$,}
\end{equation*}

while $\mathbf{A}_t$ and $\mathbf{b}_t$ for $t =1,\dots,T$ are given by:

\begin{equation*}
    (A_{kl})_t= \left \{
\begin{array}{ccc}
   \frac{\sum\limits_{j=1}^m h_{kj} \cdot (x_{ij}/\sum\limits_{a=1}^r w_{ia}^{(t-1)} \cdot h_{aj}) - w_{ik}^{(t-1)} \cdot \sum\limits_{j=1}^m h^2_{kj} \cdot (x_{ij}/(\sum\limits_{a=1}^r w_{ia}^{(t-1)} \cdot h_{aj})^2)}{\sum\limits_{j=1}^m h_{kj} + \lambda_i} \quad \text{if $l=k$},\\
 - w_{ik}^{(t-1)} \cdot \frac{\sum\limits_{j=1}^m h_{kj} \cdot (x_{ij}/(\sum\limits_{a=1}^r w_{ia}^{(t-1)} \cdot h_{aj})^2) \cdot h_{lj}}{\sum\limits_{j=1}^m h_{kj} + \lambda_i} \quad \text{if $l \neq k$};
\end{array}
\right .
\end{equation*}

$$
\begin{array}{ll}
(b_k)_t =&  
- w_{ik}^{(t-1)} \cdot \frac{\sum\limits_{j=1}^m h_{kj} \cdot (x_{ij}/(\sum\limits_{a=1}^r w_{ia}^{(t-1)} \cdot h_{aj}))}{(\sum\limits_{j=1}^m h_{kj} + \lambda_i)^2} \quad \text{for $k=1,\dots, r$} .
\end{array} 
$$

Although we focused on a specific objective function and its associated update rules, AltBi can be generalized for any $\beta$-divergence and penalty functions $\mathcal{R}$, respecting the assumptions in Section \ref{proposal}.

\begin{remark}
\label{r1}
The computational complexity of rule (\ref{H_KL}) amounts to $\mathcal{O}(Kmnr)$, where $K$ is the number of iterations. 
Update rules (\ref{multidivKL1}) and (\ref{Aandb}) are more expensive due to the use of the bunch and require $\mathcal{O}(KTmnr)$.
The complexity of other rules in Algorithm \ref{algor} is lower, which implies  $\mathcal{O}(KTmnr)$ for the whole algorithm. Note that the complexity of the proposed algorithm is larger with respect to the standard multiplicative update rules in NMF only by factor $T$. 
\end{remark}

\section{Numerical Experiments}
\label{experiments}
This section illustrates the numerical results obtained using the AltBi algorithm on two synthetic and two real datasets. It was implemented in MATLAB 2021a environment, and numerical experiments were executed on the i7 octa-core, 16GB RAM machine. 
The %tolto \url
benchmarks\footnote{{https://github.com/flaespo/Dataset\_signal\_HPO}} used in the experiments are generated according to the model\footnote{Noiseless dataset $\mathbf{Y} \in \R^{n \times m}$ was constructed. Since our goal is to solve the identification problem, it is unnecessary to perturb matrix $\mathbf{Y}$. In this way, we preserve initial sparsity.}
$\mathbf{X \approx Y = WH}$.
    
The datasets used are described in the following:
\begin{itemize}
    \item[A)] Factor matrices were generated randomly as full rank uniformly distributed matrices. Matrix $\mathbf{H} \in \R_{+}^{r \times m}$ was generated using the MATLAB command {\fontfamily{pcr}\selectfont rand}, while $\mathbf{W}\in \R_{+}^{n \times r}$ was generated using the command {\fontfamily{pcr}\selectfont randn} to obtain sparse columns. 
    Negative entries were replaced with a zero-value. 
    \item[B)] Each column in $\mathbf{W}$ is expressed as a sinusoidal wave signal with the frequency and the phase set individually for each component/column. The example of this signal waveform is plotted in Figure \ref{wave}. The negative entries are replaced with a zero-value. Factor matrix $\mathbf{H}$ was randomly generated as a full rank sparse matrix with sparseness level $\alpha_H$ adjusted by the user.
    \item[C)] The source signals from the file {\fontfamily{pcr}\selectfont
    AC10\_art\_spectr\_noi} of MATLAB toolbox NMFLAB for Signal Processing~\cite{66zdu} have been used. These signals form matrix $\mathbf{W} \in \R^{n \times r}_+$. Exemplary five signals for $n = 1000$ are plotted in Figure \ref{wave2}. Also, in this case, $\mathbf{H}$ was generated as a sparse matrix with $\alpha_H$ fixed sparsity level.
    \item[D)] Real reflectance signals taken from the U.S. Geological Survey (USGS) database  
    have been used as endmembers to generate the mixtures modelling real hyperspectral imaging data. 
    Using the NMF model, the aim is to perform hyperspectral unmixing to obtain spectral components and their corresponding proportion maps called abundances. In our approach, the column vectors of $\mathbf{W}$ contain the spectral signatures (endmembers) (Figure \ref{wave4}), and $\mathbf{H}$ represents the mixing matrix or vectorized abundance maps (Figure \ref{wave3}). The spectral signals are divided into 224 bands covering the range of wavelengths from $400$ $nm$ to $2.5$ $\mu m$. The angle between any pair of the signals is greater than $15$ degrees. These signals form matrix $\mathbf{W} \in \R^{n \times r}_+$, where $n = 224$. The rank of factorization $r$ determines the number of endmembers. 
\end{itemize}

\begin{figure}[ht]
	\centering
	\subfloat[]{\label{wave} \includegraphics[width=0.49\textwidth]{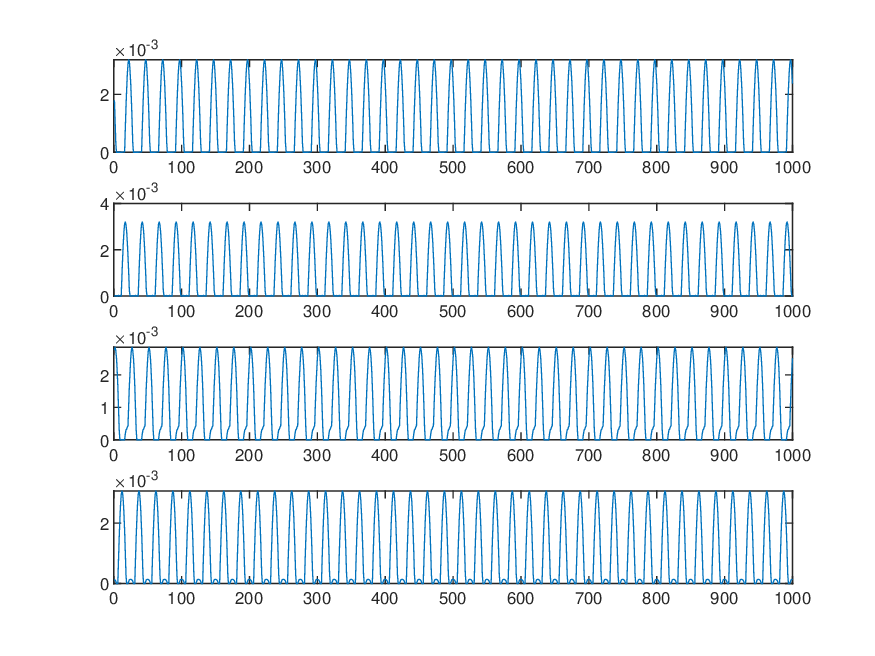}}
	\subfloat[]{\label{wave2} \includegraphics[width=0.49\textwidth]{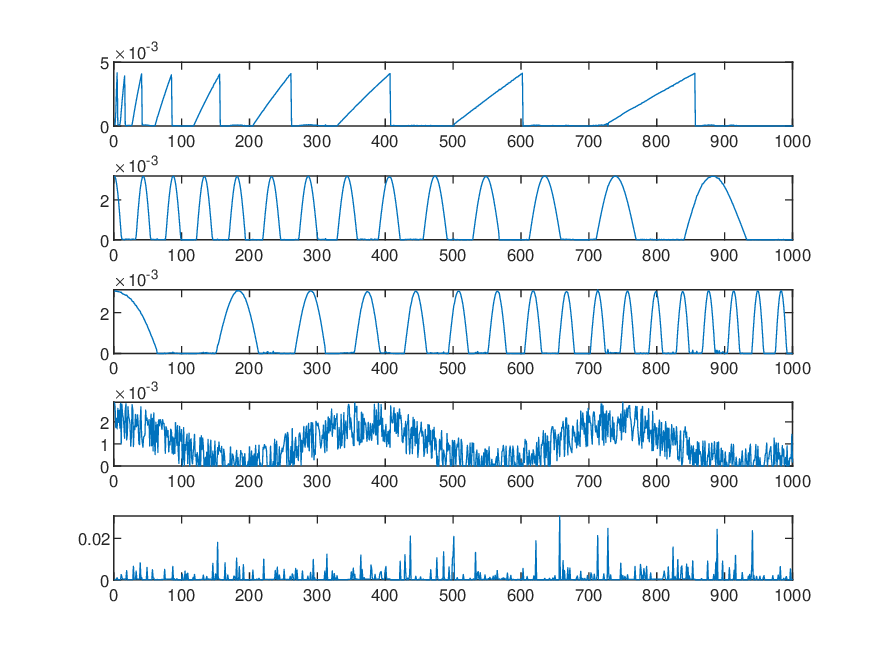}}\\
	\subfloat[]{\label{wave3} \includegraphics[width=0.49\textwidth]{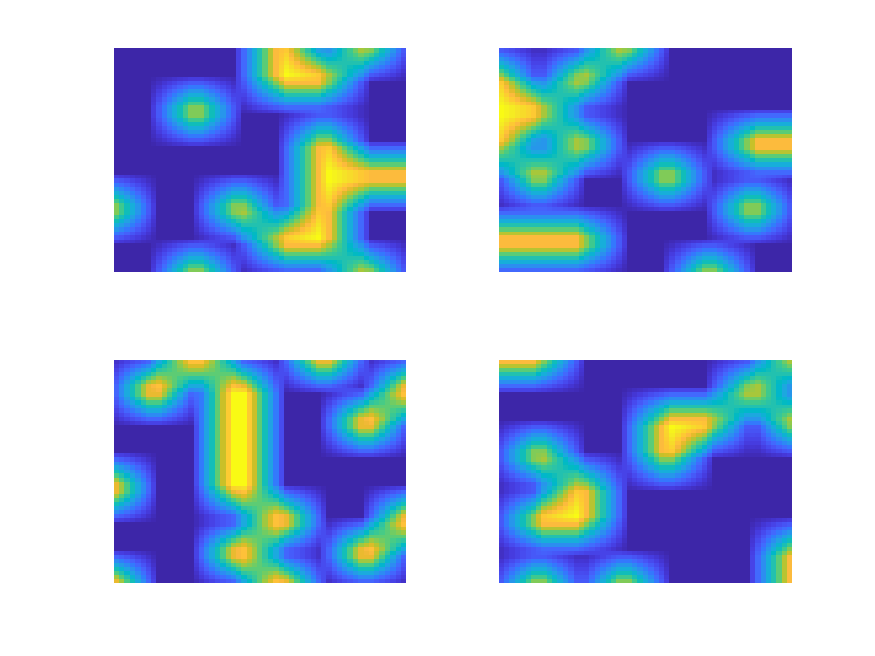}}
	\subfloat[]{\label{wave4} \includegraphics[width=0.49\textwidth]{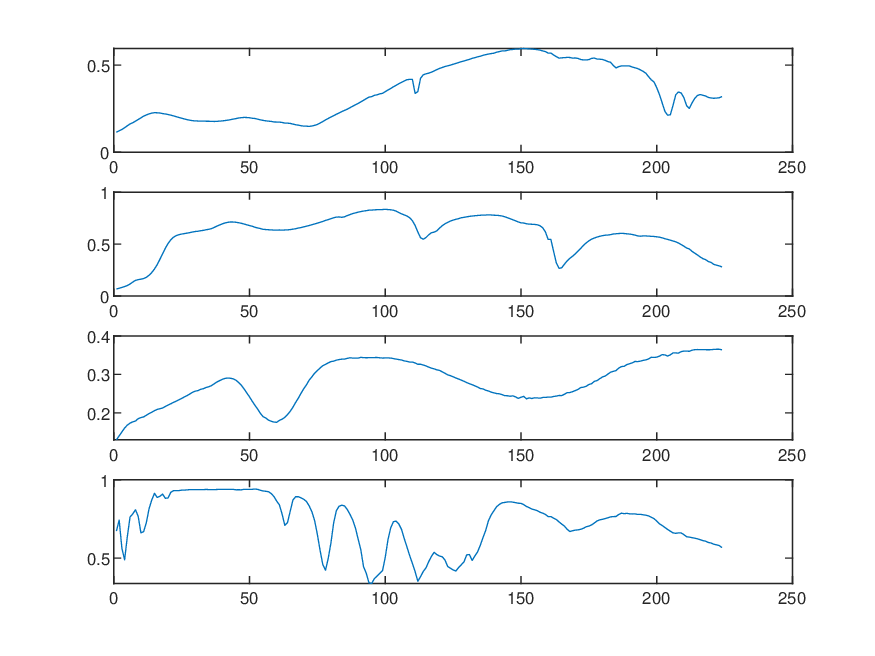}}
	\caption{Waveform of signals in benchmark B $(a)$, and benchmark C $(b)$, Abundance maps $(c)$ and Spectral signatures $(d)$ of benchmark D.}
\end{figure}

Three NMF algorithms were used and compared: AltBi, the standard un-penalized MU in~\cite{lee1999learning}, and the standard penalized that alternates rule \eqref{H_KL} and the modified version of \eqref{eq:upW} in which $\lambda_i = 0.5 \quad \forall i =1, \dots, n$
(referred to P-MU). 
The same random initializer generated from a uniform distribution starts all the algorithms \cite{esposito2021review}. The efficiency of the methods is analyzed by performing 30 Monte Carlo (MC) runs for the NMF algorithms, where for each run, initial matrices $\mathbf{W}$ and $\mathbf{H}$ are different. At the beginning of the process, initial $\bm{\lambda}$ is chosen to have homogeneity between the terms characterizing the objective function, according to:

\begin{equation*}
    \lambda_i = \frac{\sum\limits_{j=1}^m d_{1}(\mathbf{x}_j, \sum\limits_{k=1}^r w_{ik} h_{k j})}{10\cdot \mathcal{r}(\mathbf{w}_i)} \qquad \text{for}\quad i=1, \dots, n;
\end{equation*}

where $\mathcal{r}$ is the $\ell_1$ penalty norm in this particular experimental case. The maximum number of iterations for all the algorithms was set to $1000$, the tolerance for early termination to $10^{-6}$, and the number of inner iterations (length of the bunch) to 4, i.e., $T = 4$.
The following tests were performed:
\begin{itemize}
    \item[1)] Benchmark A was used with $n=1000$, $m=50$, $r=4$.
    \item[2)] Benchmark B was used with
    $n=1000$, $m=50$, $r=4$, $\alpha_H =0.1$.
    \item[3)] Benchmark C was used with 
    $n=1000$, $m=50$, $r=5$, $\alpha_H =0.1$.
    \item[4)] Benchmark D was used with 
    $n=224$, $m=3025$, $r=5$.
\end{itemize}
In all the tests, no noisy perturbations were used.

To evaluate the goodness of the approximation and the effectiveness of the minimization process, we report the relative error\footnote{In this case, we compute the 
relative error as $D_{1}(\mathbf{X, WH})/ \sum_{i,j}x_{ij}log(x_{ij})$ \cite{esposito2019orthogonal}.} %reports the Relative Error
    %$D_{1}(\mathbf{X, WH})/ %\sum_{i,j}x_{ij}log(x_{ij})$
and the evolution of the objective function with respect to iterations for benchmark A in
Figure \ref{RESS}. All other benchmarks present similar results as reported in Section \ref{appendix2}.
\begin{figure}[htb]
	\centering
	\subfloat[]{\label{resA} \includegraphics[width=0.49\textwidth]{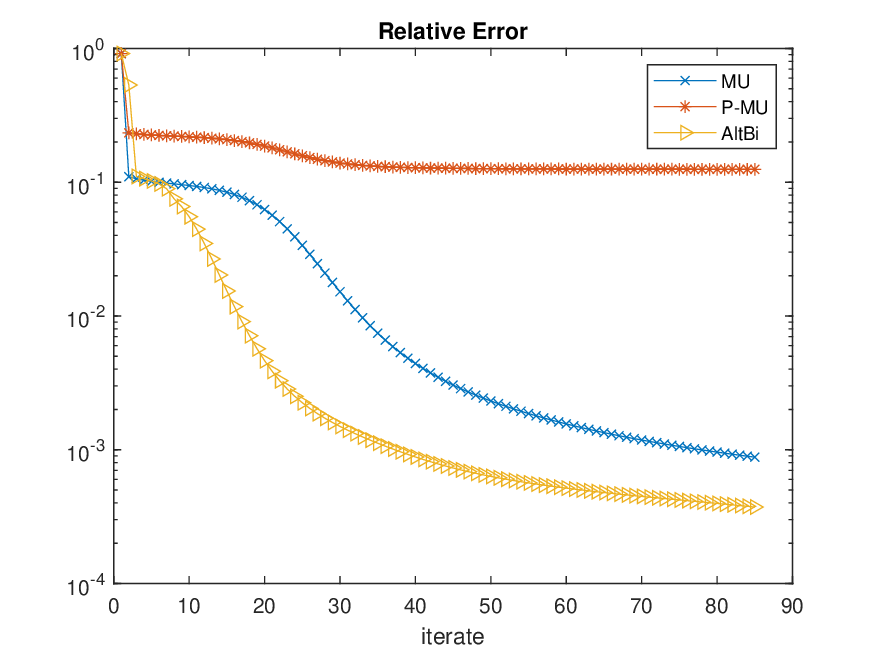}}
	\subfloat[]{\label{objA} \includegraphics[width=0.49\textwidth]{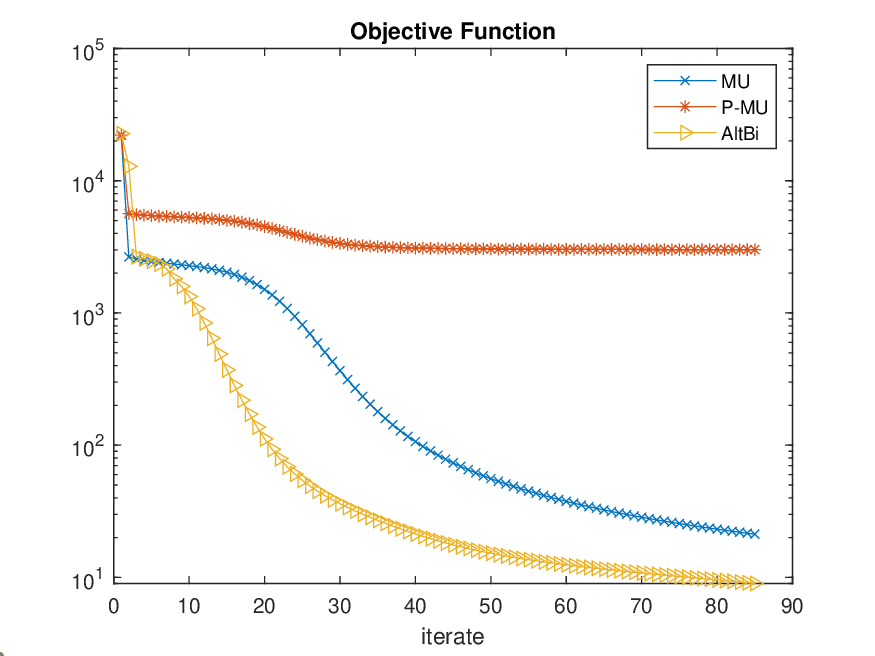}}\\
	\caption{$(a)$ Relative error, $(b)$ evolution of objective function with respect to iterations (Benchmarks A).}\label{RESS}
\end{figure}

The performance of the NMF algorithms was evaluated with the Signal-to-Interference Ratio (SIR) measure \cite{cichocki2009nonnegative} between the estimated signals and the true ones. Figure \ref{SIR totali} shows the SIR statistics (in dB) for assessing the columns in $\mathbf{W}$ and the rows in $\mathbf{H}$ for benchmark A.
\begin{figure}[htb]%[!ht]
	\centering
	\includegraphics[width=0.58\textwidth]{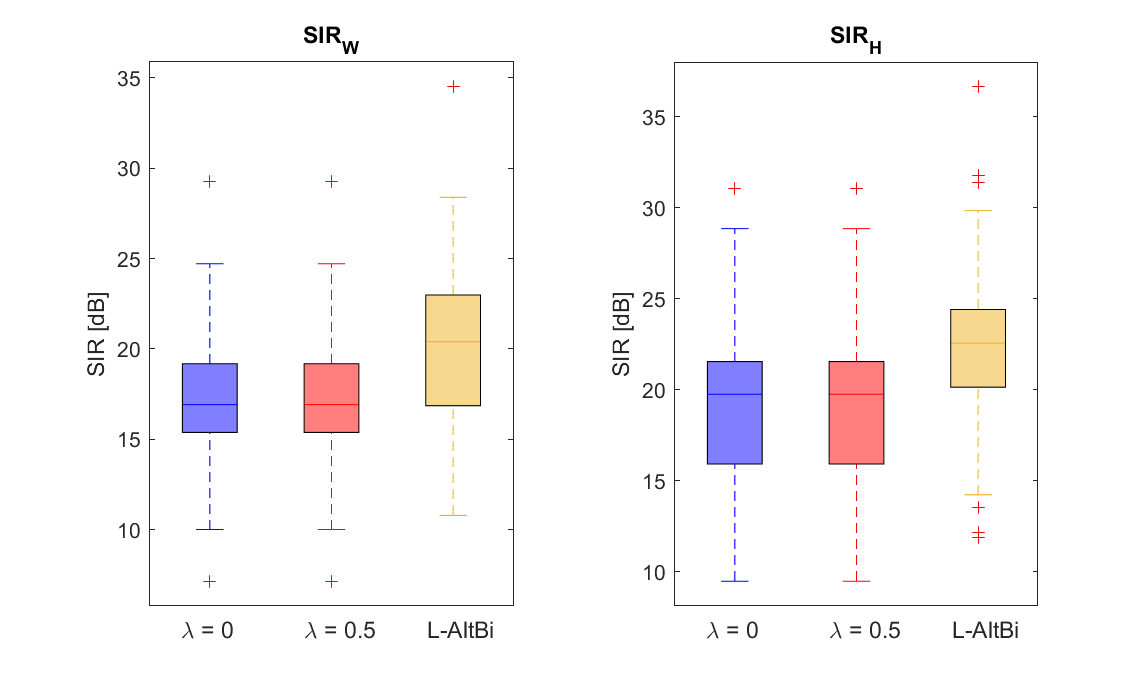}
	\caption{SIR statistics for estimating columns of $\mathbf{W}$ and rows $\mathbf{H}$ (Benchmark A).}\label{SIR totali}
\end{figure}
Table \ref{tabella_SIR} reports the numerical results of Mean-SIR in estimating $\mathbf{W}$ and $\mathbf{H}$ for benchmark A.

\begin{table}[htb]
{\footnotesize\caption{Mean-SIR [dB] for estimating matrices $\mathbf{W}$ and $\mathbf{H}$.}
\label{tabella_SIR}
\begin{center}
 \resizebox{0.6\textwidth}{!}{
\begin{tabular}{ c|c|c|c}
 &  MU & P-MU & AltBi \\
\hline
SIR for $\mathbf{W}$ & 16.7325 & 16.7325 & 21.3388\\
SIR for $\mathbf{H}$ & 19.2147 & 19.2147 & 23.3308\\
\end{tabular}
}\end{center}}
\end{table}

The general structure of the optimized $\bm{\lambda}$ has also been inspected. Figure \ref{lambda plot} compares final and initial HPs for benchmark A: pointwise and distribution of vector $\bm{\lambda}$, in Figures \ref{lambdaplot} and \ref{density}, respectively. The peak of the distribution of initial HPs shifts its location from a positive towards the zero value. Thus, the optimized $\bm{\lambda}$ is a sparse vector, suggesting the algorithm prefers to penalize the selected rows of $\mathbf{W}$ rather than all. Finally, the numerical results are also compared to evaluate the sparsity of $\mathbf{W}$ and $\mathbf{H}$ by 
%\begin{equation}\label{sparsmeasure}
   $ \text{Sp}(\mathbf{A})\footnote{{Sp}(\textbf{A}) represents the ratio between the complement of the number of elements greater than a certain threshold and the total number of elements in matrix $\textbf{A}$.} = 100 \cdot \frac{(1 - \#(\mathbf{A} > \tau ))}{\#\mathbf{A}}$,
%\end{equation}
for $\mathbf{A} \in \R^{n \times m}$. % \in \R^{n \times m}$.
The sparsity constraint was added only on $\mathbf{W}$, and for benchmarks C and D, the user provided the sparsity on $\mathbf{H}$. As shown in Figure \ref{SPA}, the proposed method enforces the sparsity on $\mathbf{W}$ and does not affect the sparsity profile in $\mathbf{H}$, as expected. 

\begin{figure}[ht]
	\centering
	\subfloat[]{\label{lambdaplot} \includegraphics[width=0.48\textwidth]{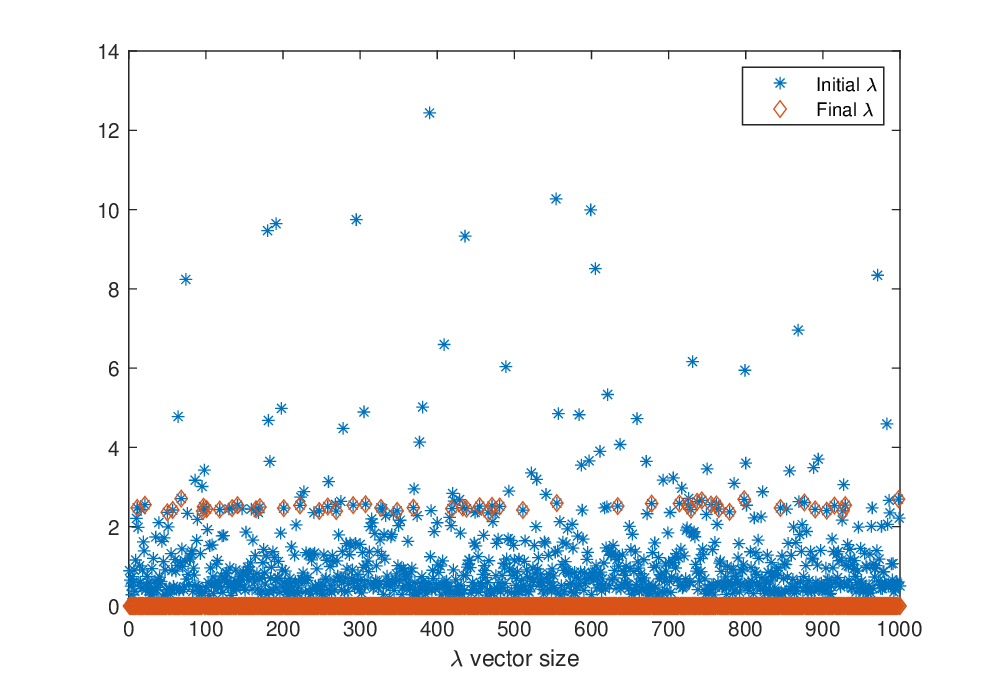}}
	\subfloat[]{\label{density} \includegraphics[width=0.58\textwidth]{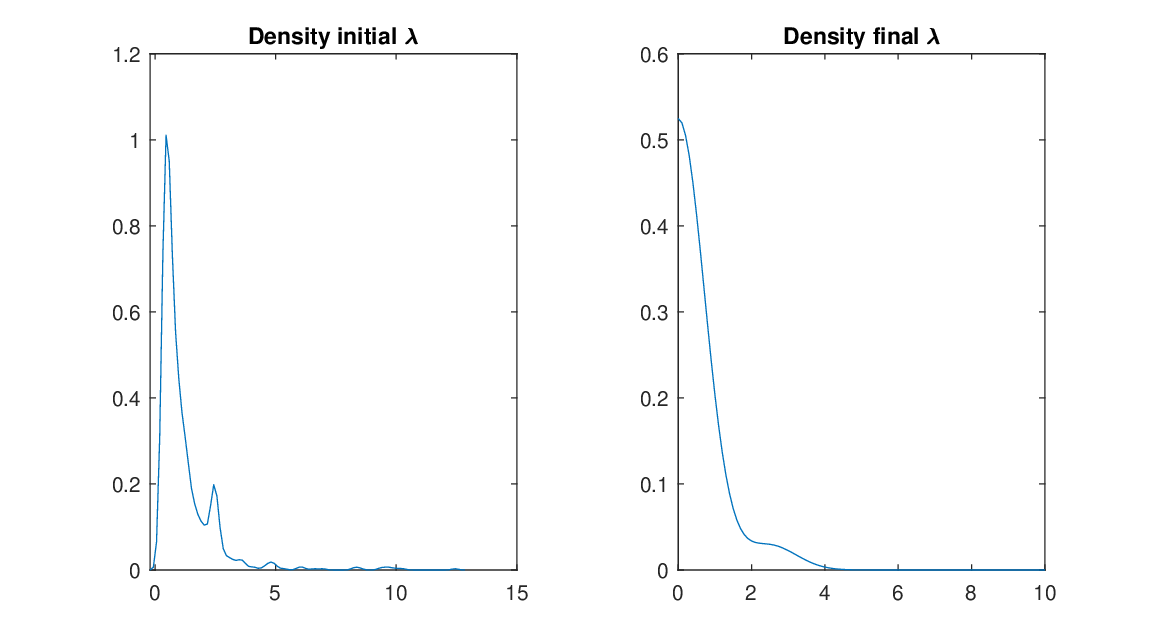}}
    \caption{Initial $\bm{\lambda}$  compared with final $\bm{\lambda}$: vector components $(a)$; density plot of $\bm{\lambda}$ vector $(b)$ (Benchmark A).}
    \label{lambda plot}
\end{figure}

\begin{figure}[ht]
	\centering
	\includegraphics[width=0.58\textwidth]{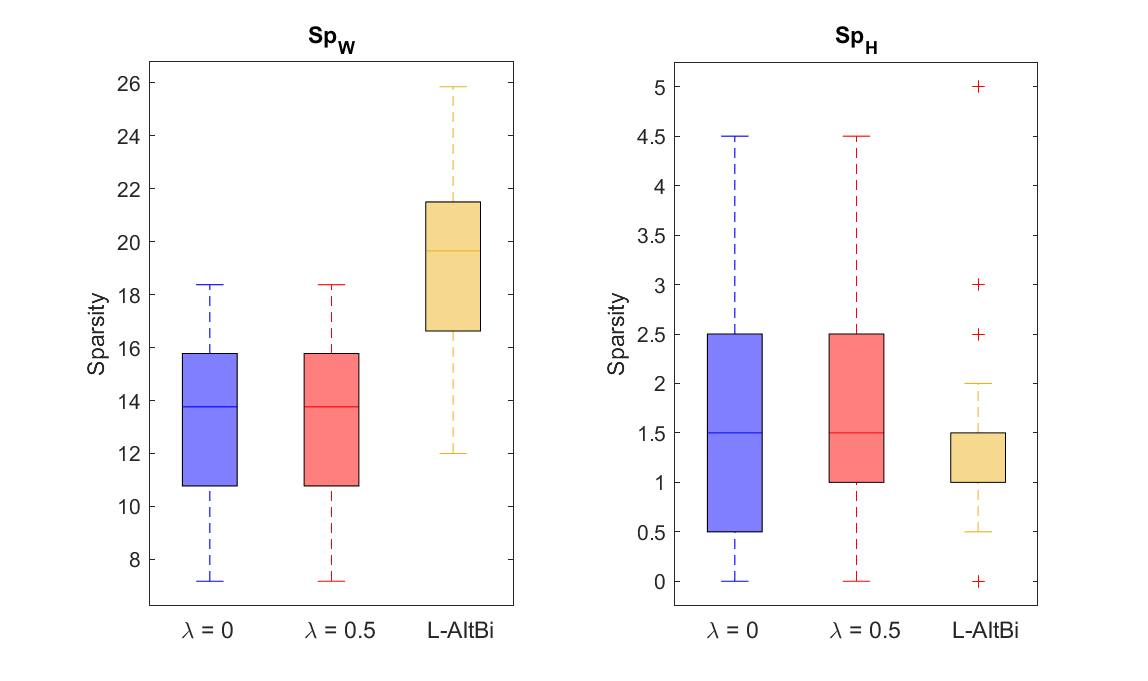}
	\caption{Statistics of the sparseness measure for $\mathbf{W}$ and $\mathbf{H}$ (Benchmark A).}\label{SPA}
\end{figure}

Please observe that optimal $\bm{\lambda}$ obtained from AltBi gives the best results either for identification and fitting problems and its choice is automatic. Figure \ref{grid} depicts the behavior of the response function for fixed values of $\lambda$ in the P-MU algorithm compared with the non-penalized MU and AltBi. AltBi shows the best performance.

\begin{figure}[ht]
    \centering
    \includegraphics[width=9cm, height=5cm]{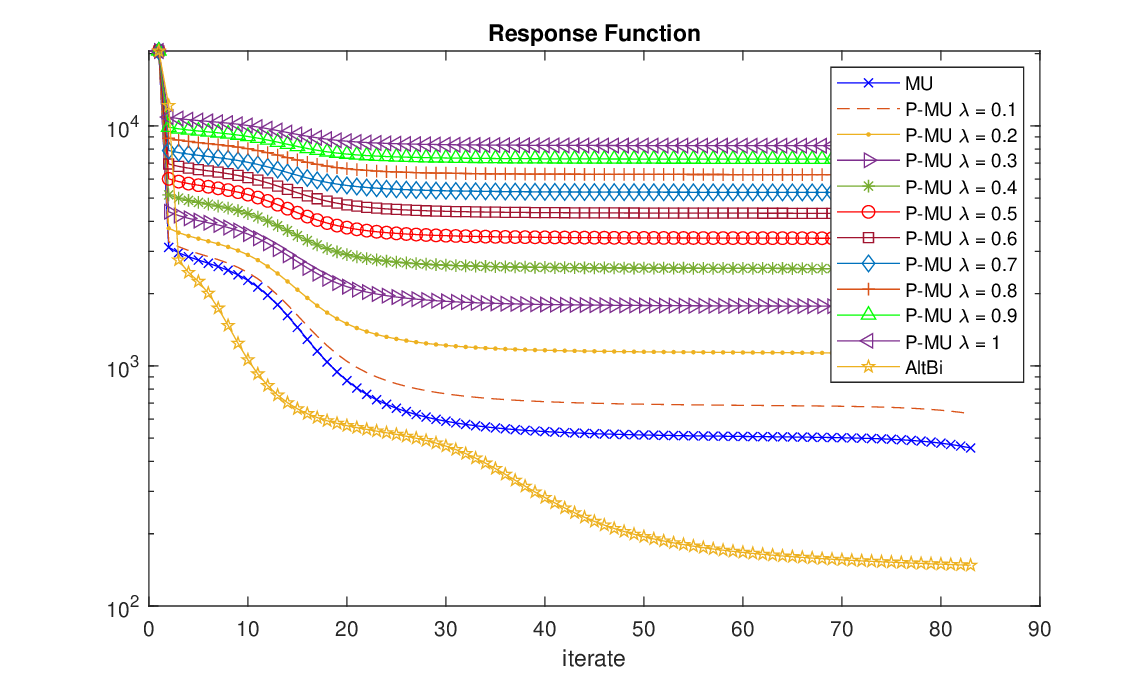}
    \caption{Response functions obtained through the P-MU algorithm with different $\lambda$ values in $\{0.1,0.2,0.3,0.4,0.5,0.6,0.7,0.8,0.9,1\}$ compared with the unpenalized MU case and AltBi.}
    \label{grid}
\end{figure}

\subsection{Results for benchmarks B, C, and D} \label{appendix2}
All the experiments confirmed the expected behavior of AltBi in terms of the identification problem. Figures \ref{sir4}, \ref{sir6}, and \ref{sir8} 
show that the SIR values obtained with AltBi are better than those obtained with MU and P-MU. 
%Resp and Obj dataset4
\begin{figure}[htb]
	\centering
	\subfloat[]{\label{res4} \includegraphics[width=0.48\textwidth]{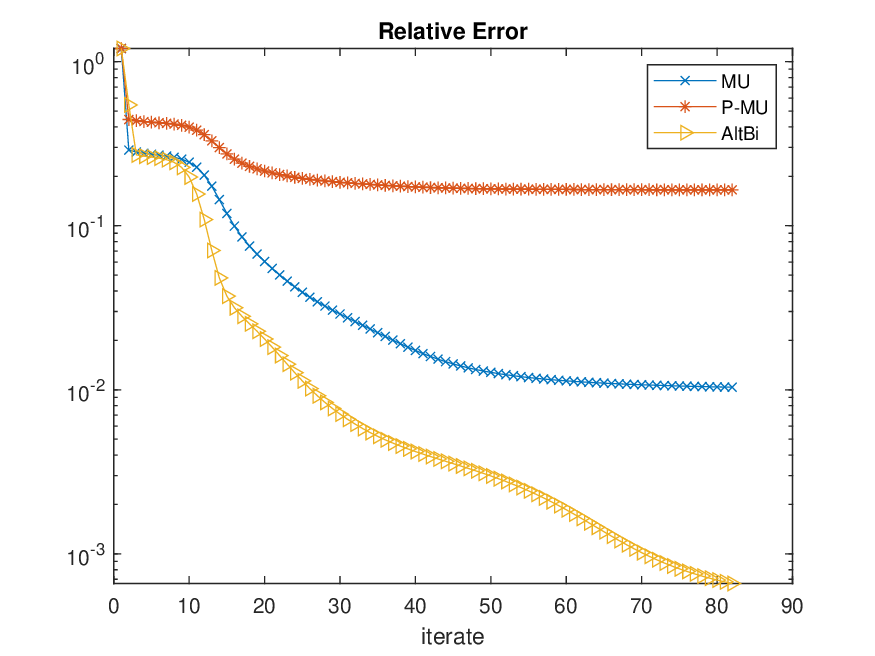}}
	\subfloat[]{\label{obj4} \includegraphics[width=0.48\textwidth]{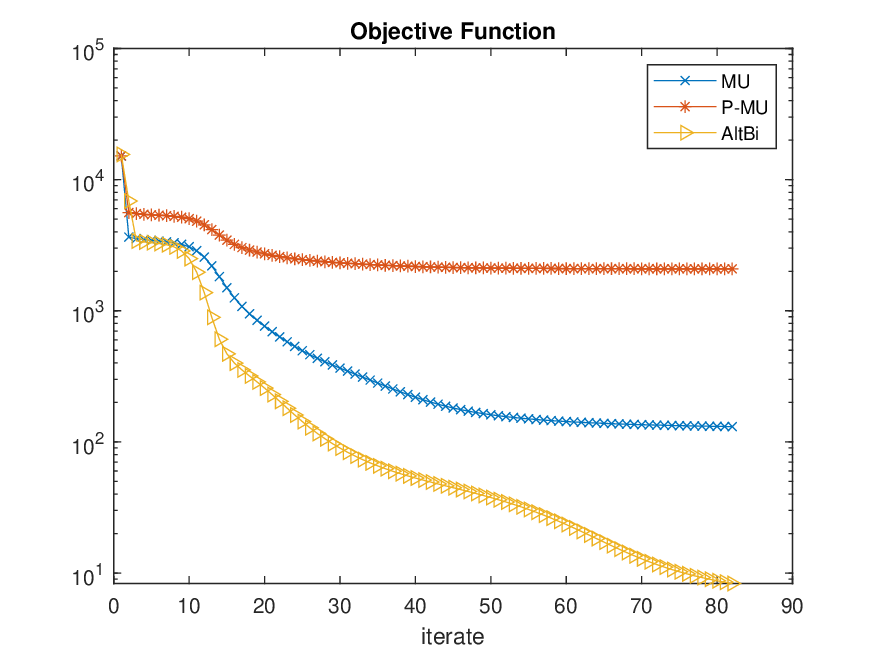}}\\
	\subfloat[]{\label{sir4} \includegraphics[width=0.48\textwidth]{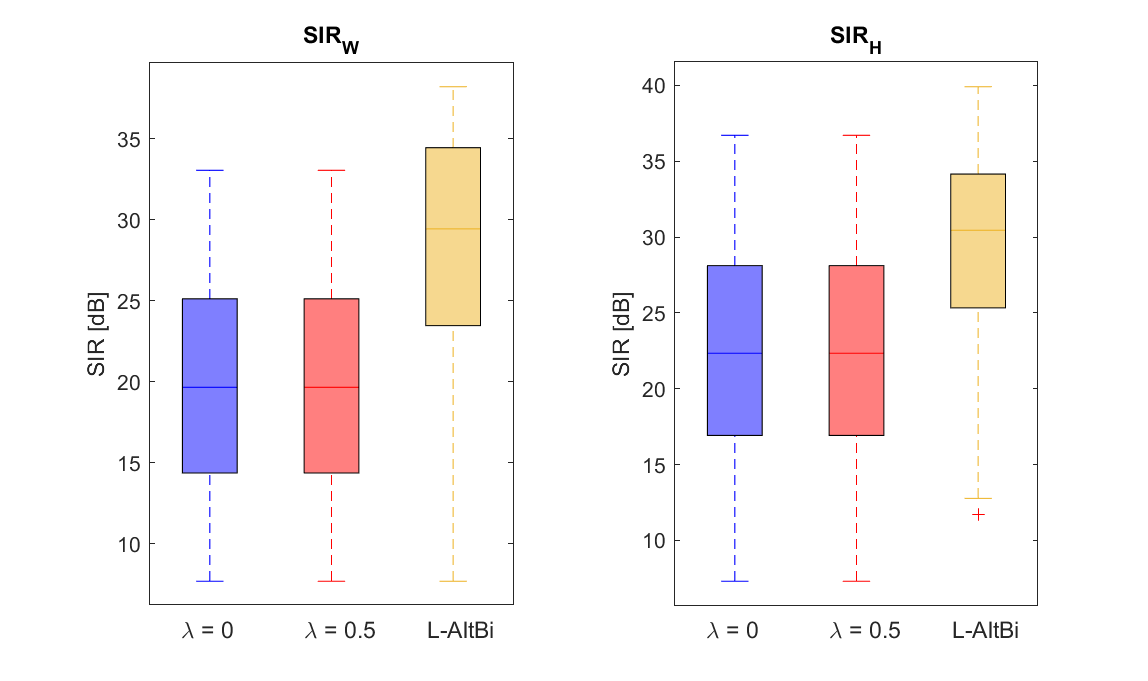}}
	\subfloat[]{\label{sp4} \includegraphics[width=0.48\textwidth]{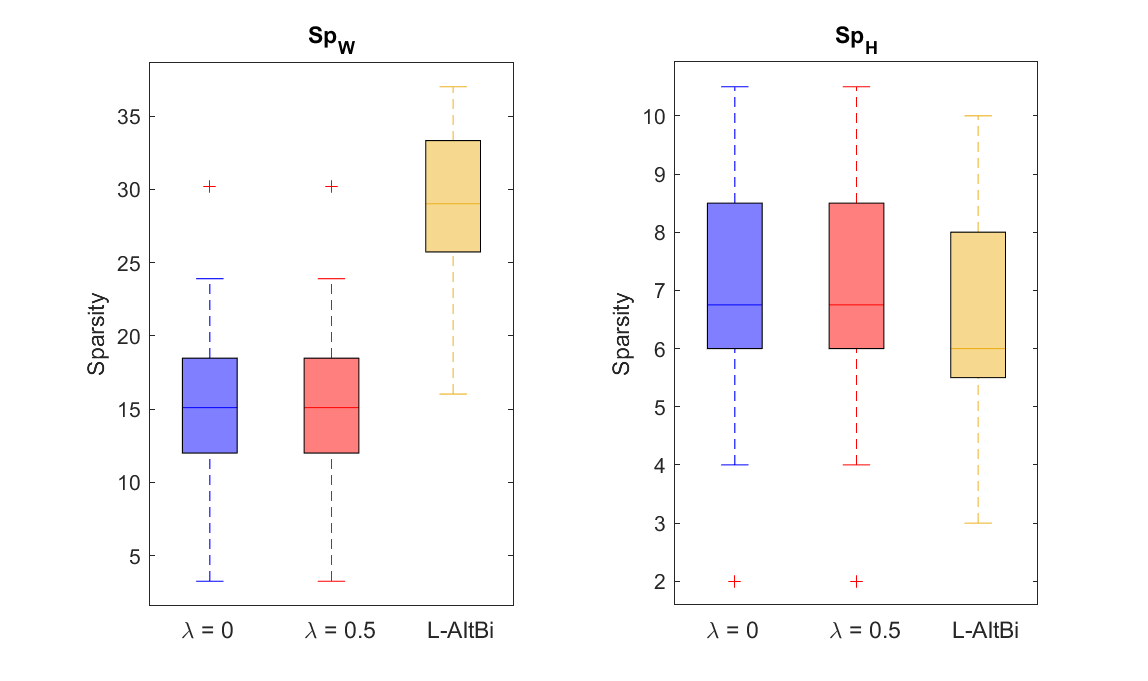}}
	\caption{$(a)$ Relative error and $(b)$ evolution of objective function with respect to iterations;  $(c)$ SIR statistics for estimating the columns of $\mathbf{W}$ and the rows of $\mathbf{H}$; $(d)$ Statistics of the sparseness measure in Benchmark B.}\label{B}
\end{figure}
%%%%%%%%%%%%%%%%%%%%%%%%%%%%%%%%%%%%%%%%%%%%%%%%
%Resp and Obj dataset 6
\begin{figure}[ht]
	\centering
	\subfloat[]{\label{res6} \includegraphics[width=0.48\textwidth]{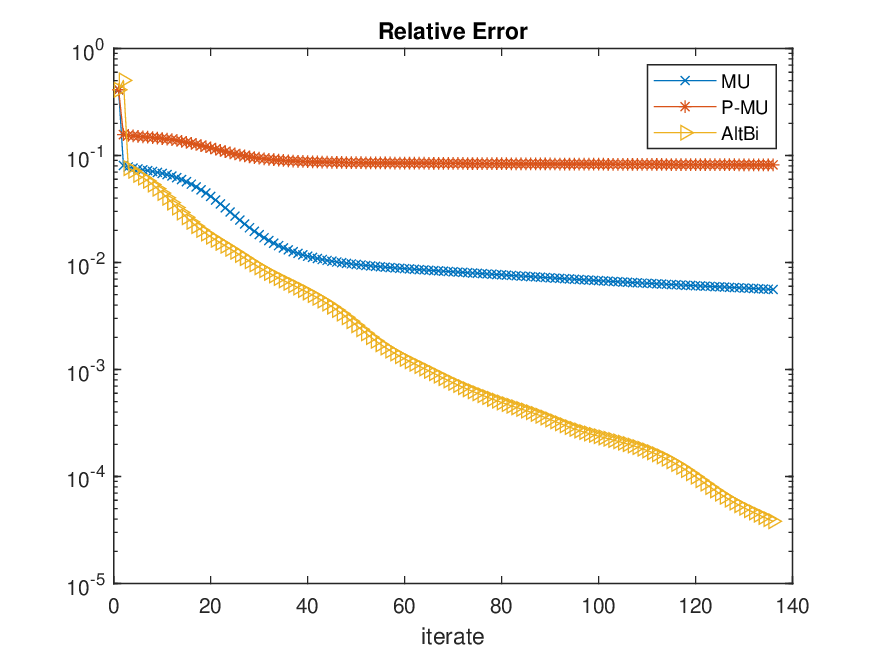}}
	\subfloat[]{\label{obj6} \includegraphics[width=0.48\textwidth]{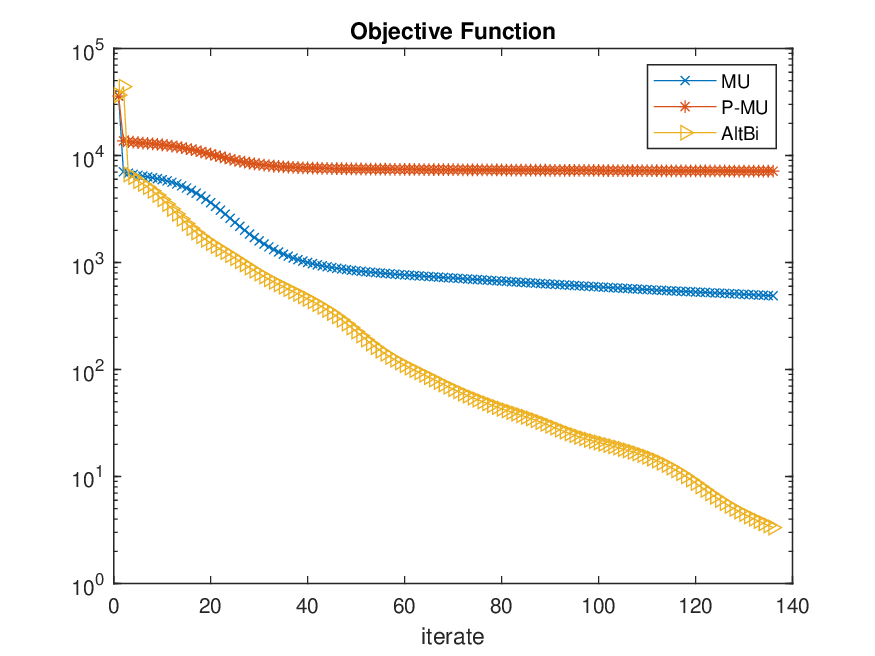}}\\
	\subfloat[]{\label{sir6} \includegraphics[width=0.48\textwidth]{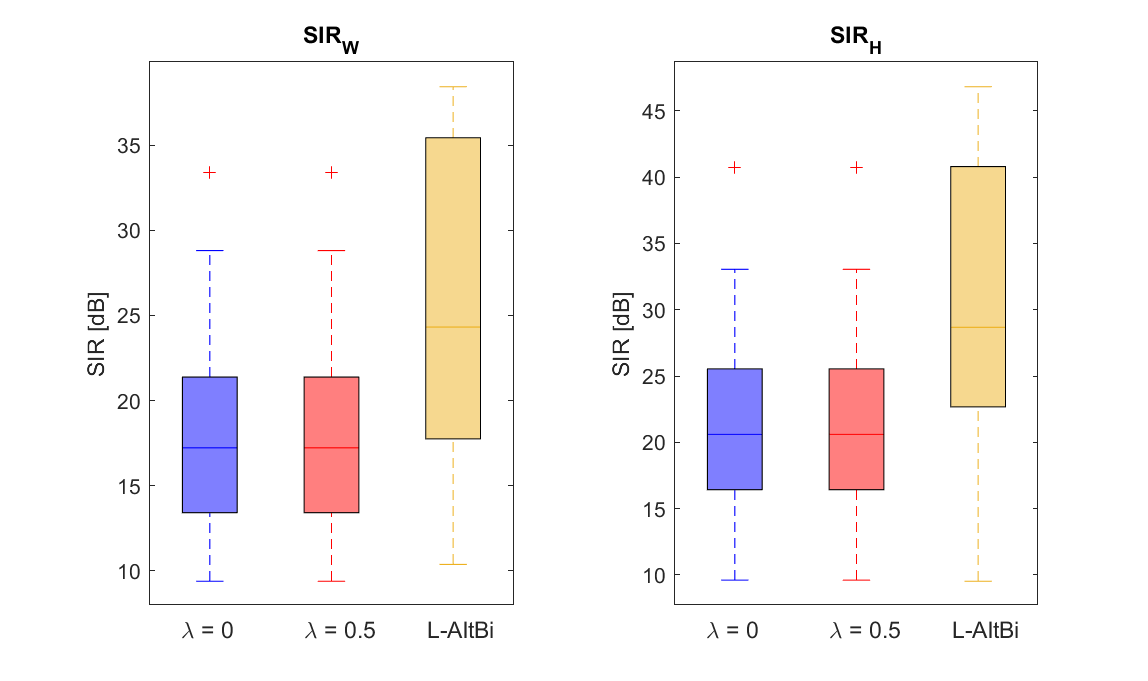}}
	\subfloat[]{\label{sp6} \includegraphics[width=0.48\textwidth]{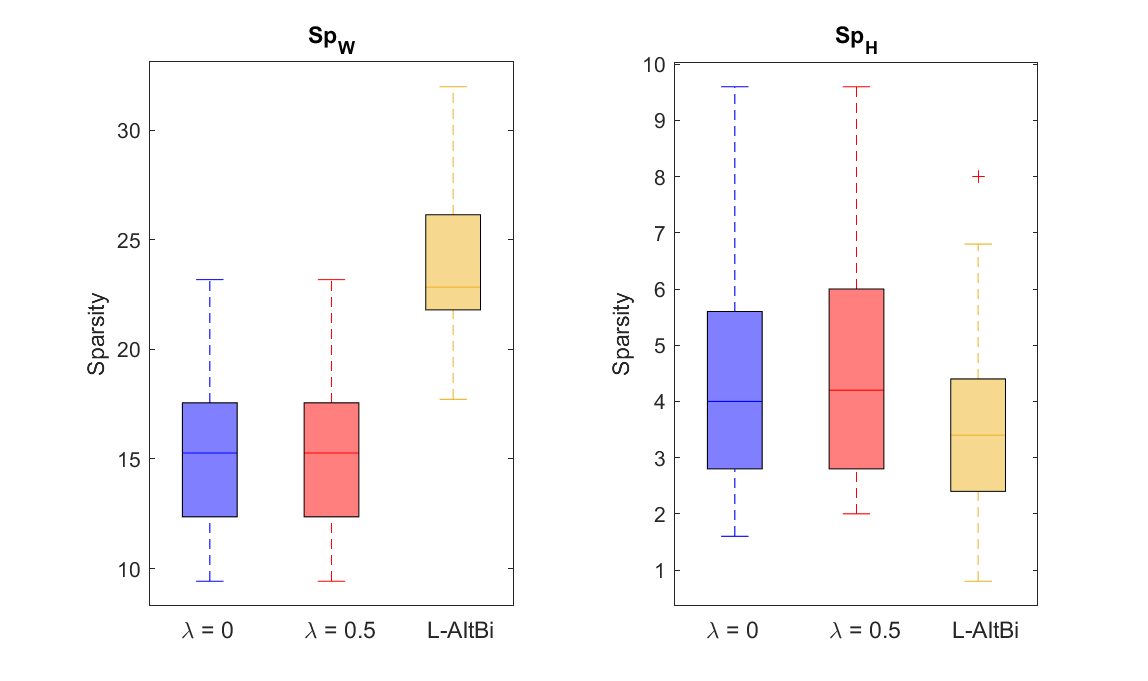}}
	\caption{$(a)$ Relative error and $(b)$  evolution of objective function with respect to iterations;  $(c)$ SIR statistics for estimating the columns of $\mathbf{W}$ and the rows of $\mathbf{H}$; $(d)$ Statistics of the sparseness measure  in Benchmark C.}\label{C}
\end{figure}
%%%%%%%%%%%%%%%%%%%%%%%%%%%%%%%%%%%%%%%%%%%%%%%%%%
%Resp and Obj dataset8
\begin{figure}[ht]
	\centering
	\subfloat[]{\label{sir8} \includegraphics[width=0.48\textwidth]{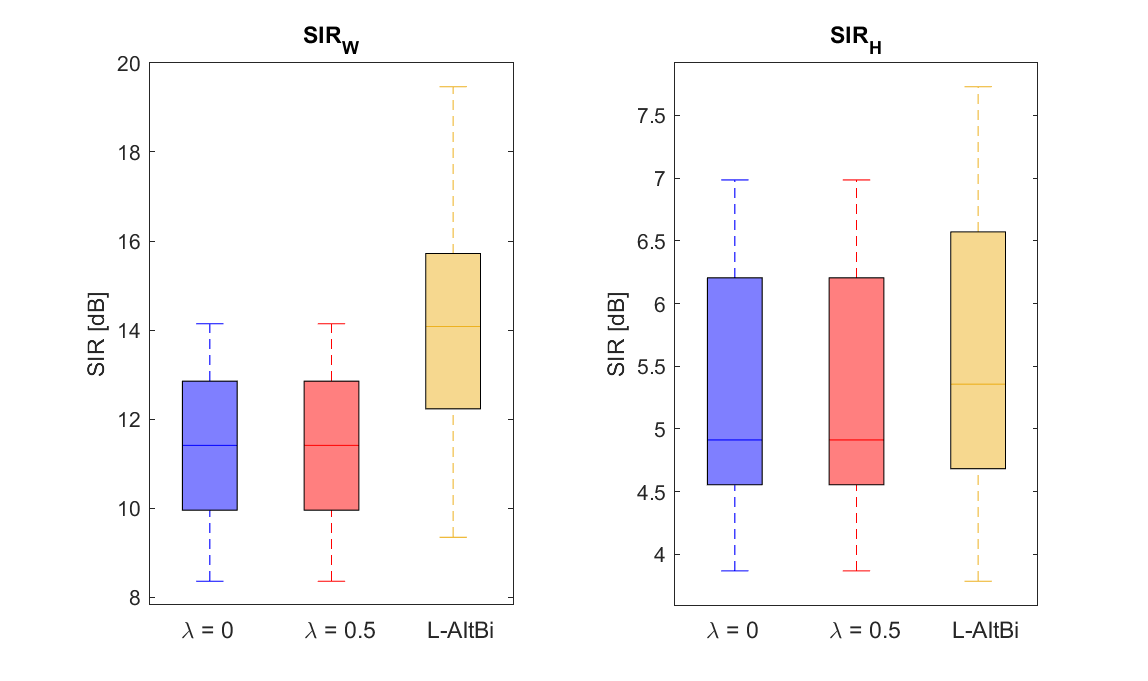}}
	\subfloat[]{\label{sp8} \includegraphics[width=0.48\textwidth]{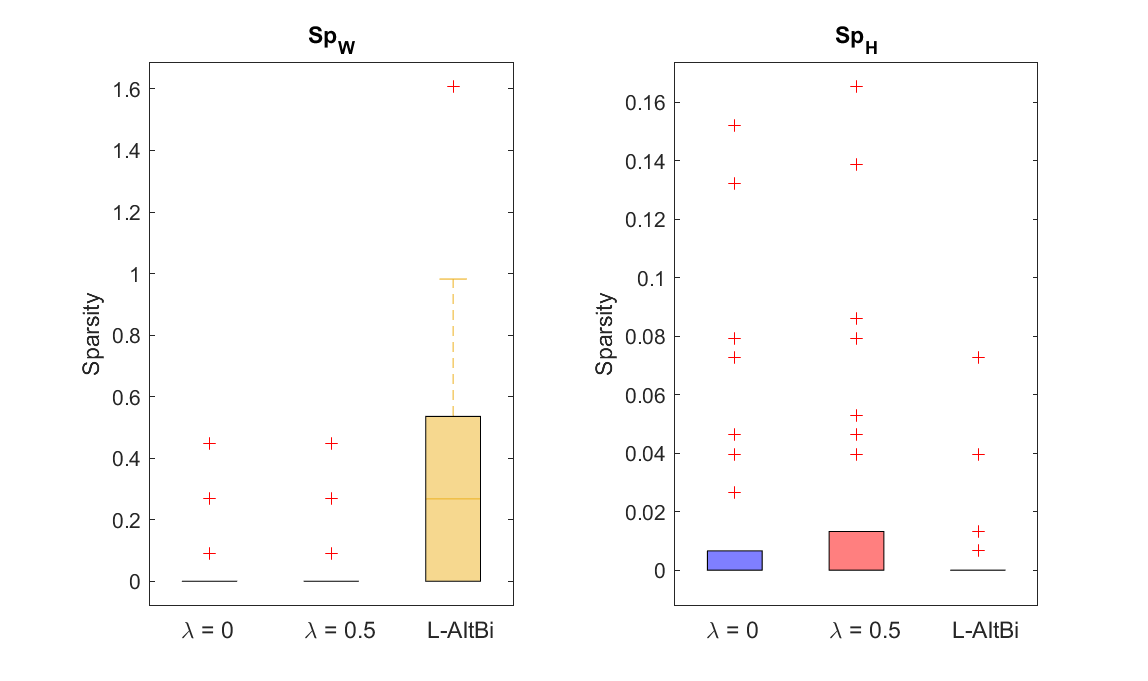}}
	\caption{SIR statistics for estimating the columns of matrix $\mathbf{W}$ (spectral signatures) and the rows of matrix $\mathbf{H}$ $(a)$; Statistics of the sparseness measure $(b)$ in Benchmark D.}\label{D}
\end{figure}
Moreover, for benchmark D, we show the original abundance maps (\ref{Initial abundance}) and spectral signatures (\ref{Initial sign}) compared to the estimated abundance maps (\ref{Final abundance}) and spectral signatures (\ref{Final sign}). Similar results are obtained for the relative error and the objective function in benchmark D, which we omit for brevity. 
The abundance maps are estimated with lower SIR performance than the spectral signatures (matrix $\mathbf{W}$). This result is not surprising: no penalty is imposed on $\mathbf{H}$. The sparsity-enforcing term was considered only for estimating matrix $\mathbf{W}$. 
\begin{figure}[ht]
	\centering
	\subfloat[Original abundance maps ]{\label{Initial abundance} \includegraphics[width=0.48\textwidth]{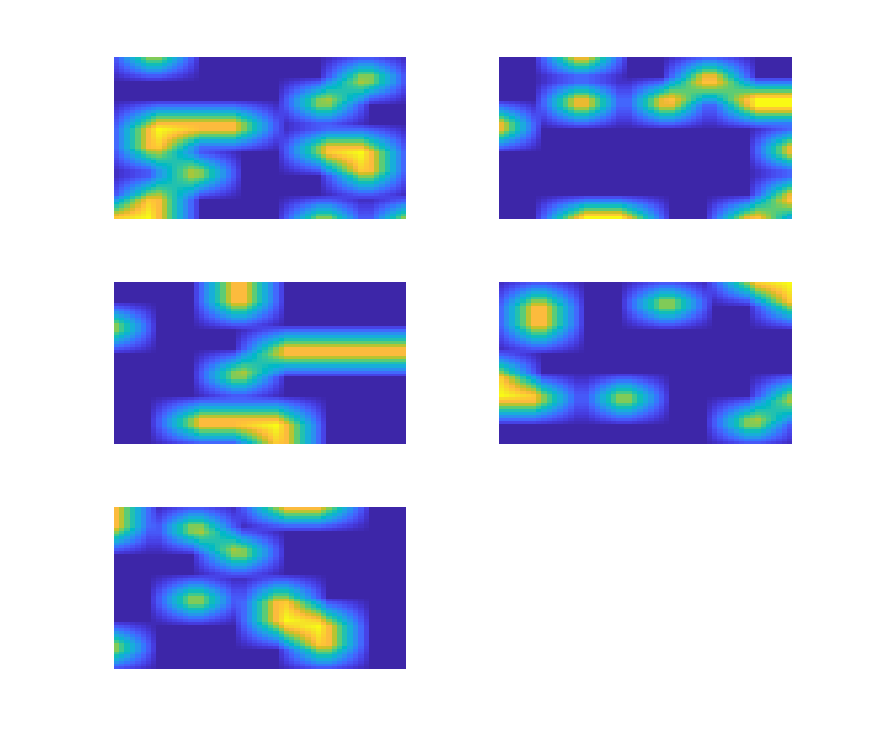}}
	\subfloat[Estimated abundance maps ]{\label{Final abundance} \includegraphics[width=0.48\textwidth]{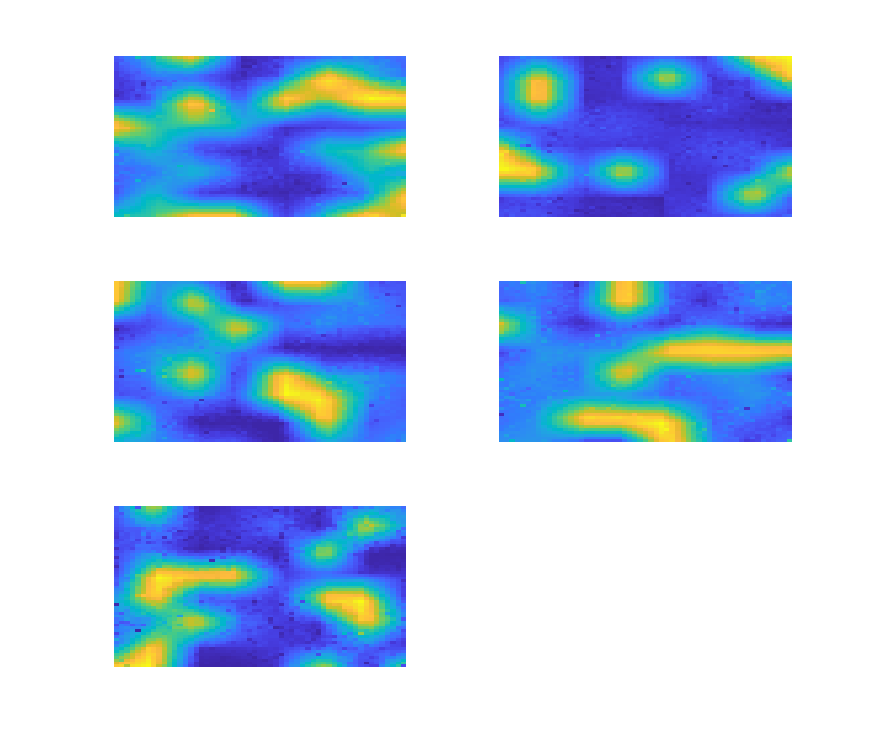}}\\
	\subfloat[Original spectral signatures ]{\label{Initial sign} \includegraphics[width=0.48\textwidth]{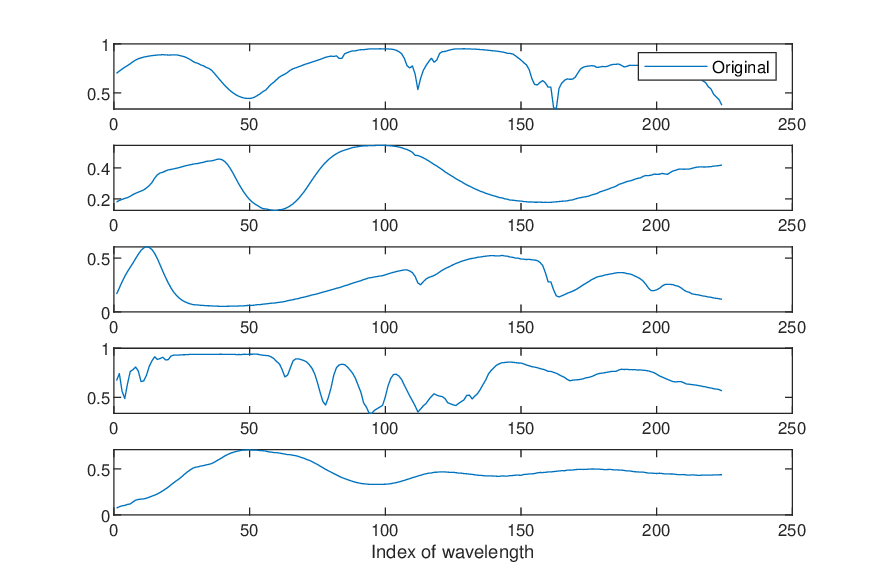}}
	\subfloat[Estimated spectral signatures ]{\label{Final sign} \includegraphics[width=0.48\textwidth]{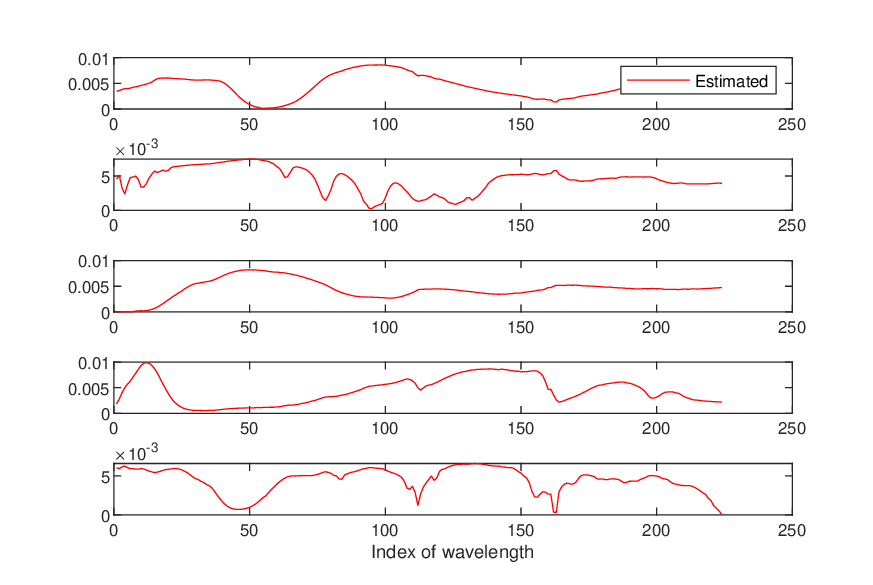}}
    \caption{Abundance maps: $(a)$ original, $(b)$ estimated with AltBi. Spectral signatures: $(c)$ original, $(d)$ estimated (in Benchmark D)}
    \label{spettral}
\end{figure}

\section{Conclusions}\label{conclusion}
We proposed the alternating HPO procedure for NMF problems which incorporates the penalty HPs into the optimization problem with the bi-level mode. We proved the existence and convergence results for the solution of the considered task and provided promising numerical experiments and comparisons.

HPO in an unsupervised scenario of data matrix factorization represents an evolving topic. However, when the size of the problem increases, the computational cost required by AltBi could not make this algorithm very competitive. To improve the computational efficiency, a column-wise version of AltBi is under study with the possibility of speeding up the algorithm by varying the length of the bunch to make the local truncation error approximately constant.

The extension of the theoretical results under hypotheses $(1)-(6)$ with no convex error and loss functions could also be considered. These aspects could accomplish this evolving topic together with the analysis of the effects made by different choices of the penalty functions on performance and computational issues for large dataset applications (such as gene expression analysis~\cite{taslaman2012framework,esposito2019orthogonal}, blind spectral unmixing~\cite{zdunek2014regularized, leplat2020multi}, and text mining).

\appendix
\section{Convergence and Correctness for the $\textbf{W}$ update in (\ref{multidivKL1})} \label{appendix}

Without loss of generality, the function in (\ref{partcase_ex}) can be rewritten neglecting constants which are not relevant to the minimization process. Thus:

\begin{equation}
\sum_{i,j}\left( -x_{ij}log\left(\sum\limits_{k=1}^r w_{ik}h_{kj}\right)+\sum\limits_{k=1}^r w_{ik}h_{kj}\right)+\sum_{i,j}\lambda_i{w_{ij}}.
\label{obj2}
\end{equation}

In particular, we theorize its element-wise update rules as:

\begin{equation}
\label{eq:upW}
w_{ia}\leftarrow w_{ia} \frac{\sum\limits_{j=1}^m{(h_{aj}x_{ij}/\sum\limits_{k=1}^r w_{ik}h_{kj})}}{\sum\limits_{j=1}^m{h_{aj}}+\lambda_i}, \quad \text{for $i=1,\dots,n$ and $a=1,\dots,r$}.
\end{equation}

Fixing the $i$-th row, let $\mathbf{w}_i \in\mathbb{R}^r$ and $ \mathbf{x}_i \in\mathbb{R}^m$ be the $i$-th rows of $\mathbf{W}$ and $\mathbf{X}$, respectively, the function in (\ref{obj2}) can be rewritten with respect to unknown $\mathbf{w}_i$ as

\begin{equation}
\mathcal{F}(\mathbf{w}_i)=\sum_{j=1}^m{-x_{ij}\log\left(\sum_{a=1}^r{w_{ia}h_{aj}}\right)}+\sum_{j=1}^m\sum_{a=1}^r{w_{ia}h_{aj}}+\lambda_i\sum_{a=1}^r{w_{ia}},
\label{eq:div}
\end{equation}

then the updates for unknown $\mathbf{w}_i$ follow from Theorem \ref{th}.

\begin{theorem}
\label{th}
The divergence in (\ref{eq:div}) is non-increasing under update rules \eqref{eq:upW}.
The divergence is invariant under these updates if and only if $\mathbf{w}_i$ is a stationary point of the divergence.
\end{theorem}
The following proof proceeds the demonstration scheme proposed by Lee and Seung~\cite{seung2001algorithms} and Liu et al~\cite{liu2003non}, but it adopts a different and more general formulation of the auxiliary function for objective function (\ref{eq:div}).

\begin{lemma}
$\mathcal{G}(\mathbf{w}_i,\mathbf{w}_i^{t})= \sum\limits_{j=1}^{m}\sum\limits_{a=1}^{r}{w_{ia}h_{aj}}$ 

\begin{align*}
&-\sum\limits_{j=1}^{m}\sum\limits_{a=1}^{r}{x_{ij}\frac{w_{ia}^th_{aj}}{\sum\limits_{b=1}^r{w_{ib}^th_{bj}}}\left(log\left(w_{ia}h_{aj}\right)-log\left(\frac{w_{ia}^th_{aj}}{\sum\limits_{b=1}^r{w_{ib}^th_{bj}}}\right)\right)} +\lambda_i\sum\limits_{a=1}^r{w_{ia}}
\end{align*}

is an auxiliary function for $\mathcal{F}(\mathbf{w}_i)$.
\end{lemma}
\proof
We prove that $\mathcal{G}(\mathbf{w}_i,\mathbf{w}_i^t)$ is an auxiliary function for $\mathcal{F}(\mathbf{w}_i)$.
Due to the basic proprieties of the logarithmic function, the condition $\mathcal{G}(\mathbf{w}_i,\mathbf{w}_i)=\mathcal{F}(\mathbf{w}_i)$ is straightforward.
To prove that $\mathcal{G}(\mathbf{w}_i,\mathbf{w}_i^t)\geq \mathcal{F}(\mathbf{w}_i)$, we consider the quantity

\begin{equation}
\alpha_{aj}=\frac{w_{ia}^t h_{aj}}{\sum_b w_{ib}^t h_{bj}}\quad with\quad \sum_j\sum_a{\alpha_{aj}}=1.
\end{equation}

Due to the convexity of the logarithmic function, the inequality

\begin{equation}
\sum_j{x_{ij}\log{\sum_a{w_{ia}h_{aj}}}} -\sum_{j}\sum_{a}{x_{ij}\alpha_{aj}\log\left({\frac{w_{ia}H_{aj}}{\alpha_{aj}}}\right)}\geq0
\end{equation}

holds, so that the proof follows. 
\endproof

\begin{lemma}\label{lemma1} Objective function $\mathcal{F}$ is non-increasing when its auxiliary function is minimized.
\end{lemma}
\proof
The minimum value of $\mathcal{G}(\mathbf{w}_i,\mathbf{w}_i^t)$ with respect to $\mathbf{w}_i$ satisfies 

\begin{equation}
\frac{d \mathcal{G}(\mathbf{w}_i,\mathbf{w}_i^t)}{dw_{ia}}=\sum_j{h_{aj}}-\sum_j{x_{ij}\frac{w_{ia}^t h_{aj}}{\sum_b{w_{ib}^t h_{bj}}}\left(\frac{1}{h_{aj}}\right)+\lambda_i}=0.
\end{equation}

Thus, the update rule is \eqref{eq:upW}.
\endproof
According to this new update, the KKT conditions with respect to the nonnegative constraints are:

\begin{equation}
\begin{cases}
\mathbf{W}.*\nabla_{\mathbf{W}}\mathcal{F}(\mathbf{W},\mathbf{H})=0,\\
%H.*\nabla_HF(W,H)=0,\\
\nabla_{\mathbf{W}}\mathcal{F}(\mathbf{W,H})\geq 0,\\
%\nabla_HF(W,H)\geq 0,\\
\mathbf{W}\geq 0,\\
%H\geq 0,
\end{cases}
\label{KKT}
\end{equation}

where $.*$ is the Hadamard pointwise product and $\nabla_{\mathbf{W}}$ %and $\nabla_H$ are 
is the gradient of (\ref{obj2}). This formulation allows to prove that update (\ref{eq:upW}) satisfies KKT conditions (\ref{KKT}) at the convergence, then its correctness is ensured.

\section*{Acknowledgments}
We would like to thank Prof. N. Gillis from University of Mons and Prof. C. Kervazo from Télécom Paris Institut Polytechnique de Paris for their remarks during the debate on the optimization problem.\\ 
N. D. B., F. E., and L.S. were supported in part by the GNCS-INDAM (Gruppo Nazionale per il Calcolo Scientifico of Istituto Nazionale di Alta Matematica) Francesco Severi, P.le Aldo Moro, Roma, Italy.\\ 
F.E. is supported by REFIN Project 363BB1F4.% Reference project idea UNIBA027 ``Un modello numerico-matematico basato su metodologie di algebra lineare e multilineare per l'analisi di dati genomici''.
\\This research did not receive any specific grant from funding agencies in the public, commercial, or not-for-profit sectors.\\
\textbf{Declarations of interest}: none.\\
\textbf{Authors contributions:} All authors equally contribute to this work.

%\bibliography{mybibfile}

\end{document}